\documentclass[12pt]{amsart}
\usepackage[amsthm, mlmodern]{basemath}
\usepackage{thm-restate}
\usepackage{yhmath}
\usepackage{fullpage}
\usepackage{setspace}
\usepackage{nicematrix}
\usepackage{multicol}
\usepackage{subfig}
\usepackage{float}
\usepackage{setspace}
\usepackage{caption}
\usepackage{faktor}
\usepackage{bbm}
\usepackage{bm}
\usepackage{braket}

\numberwithin{equation}{theorem}
\RequirePackage[shortlabels]{enumitem}

\DeclareMathOperator{\Pprim}{\mathrm{P.Prim}}

\DeclareMathOperator{\GKdim}{GKdim}

\newcommand{\orbit}[1]{\mathbb{O}({#1})}

\newcommand{\p}{\mathfrak{p}}

\newcommand{\Dx}{\mathrm{Dx}}

\newcommand{\bchi}{\overline{\chi}}

\newcommand{\barT}{\overline{T}}
\newcommand{\barB}{\overline{B}}
\newcommand{\barX}{\overline{X}}
\newcommand{\barC}{\overline{C}}

\newcommand{\bart}{\bar{t}}
\newcommand{\barv}{\bar{v}}

\newcommand{\UW}{\Ua(\W{-1})}

\title{Ideals of homomorphic images of the enveloping algebra of the Witt algebra}
\author{Tuan Anh Pham}
\address{School of Mathematics, The University of Edinburgh, United Kingdom}
\email{tuan.pham@ed.ac.uk}
\subjclass[2020]{17B66, 16P70, 16N60, 16D60}
\keywords{Witt algebra, universal enveloping algebra, ascending chain condition, prime spectrum, primitive spectrum}

\begin{document}
\begin{abstract}
    Let $W_{\geq -1} = \mathbb{C}[t]\partial$ and $W = \mathbb{C}[t, t^{-1}]\partial$ be the Witt algebra of algebraic vector fields on $\mathbb{C}$ and $\mathbb{C}^*$ respectively. In this paper, we make significant progress toward the open conjecture that the enveloping algebras $\mathrm{U}(W_{\geq -1})$ and $\mathrm{U}(W)$ satisfy the ascending chain condition (ACC) on two-sided ideals. We show that all homomorphic images of $\mathrm{U}(W_{\geq -1})$ and $\mathrm{U}(W)$ under the family of ``orbit homomorphisms'' of arbitrary Gelfand-Kirillov dimension satisfy ACC on ideals. These orbit homomorphisms were the key ingredient allowing us to ``lift'' the Dixmier map from finite-dimensional solvable settings to infinite-dimensional settings of the Witt and Virasoro algebras in our recent work \cite{pham2025orbit}.
    
    As a result, we completely classify the prime and primitive spectra of these homomorphic images. As these images approximate $\mathrm{U}(W_{\geq -1})$ better as their GK-dimension increases, this classification sheds new light on the two-sided and prime ideal structures of $\mathrm{U}(W_{\geq -1})$. Finally, we discuss several applications of our results to the Dixmier map for $W_{\geq -1}$.
\end{abstract}

\maketitle

\section{Introduction}
Since their appearance in string theory and conformal field theory, the Witt and Virasoro algebras have been central objects in physics and mathematics. However their ring-theoretic properties remained mysterious until Sierra and Walton proved that their enveloping algebras are neither left nor right Noetherian \cite{sierra2014universal, sierra2016maps}. Their results provided strong evidence for the celebrated open conjecture that a Lie algebra $\g$ is finite-dimensional if and only if its enveloping algebra $\Ua(\g)$ is Noetherian, and ultimately enabled Andruskiewitsch and Mathieu \cite{andruskiewitsch2025noetherian} to establish the conjecture for all simple $\ZZ^n$-graded Lie algebras. Let $\W{-1} = \Der(\CC[t])$ and $W = \Der(\CC[t, t^{-1}])$ be the \emph{one sided} and \emph{(full)} Witt algebra respectively. The key ingredient in \cite{sierra2014universal, sierra2016maps} is an interesting homomorphic image of $\Ua(\W{-1})$ of Gelfand-Kirillov(GK) dimension $3$ coming from algebraic geometry. They showed that although that image is not Noetherian, it satisfies the ascending chain condition (ACC) on two-sided ideals and conjectured the same holds for $\Ua(\W{-1})$. 
\begin{conjecture}[{\cite{sierra2014universal}}]\label{conj: UW ACC}
    The enveloping algebra $\Ua(\W{-1})$ satisfies ACC on two-sided ideals.
\end{conjecture}
There has been some weak evidence for \Cref{conj: UW ACC} including Iyudu and Sierra's result in \cite{iyudu2020enveloping} that nonzero ideals of $\Ua(\W{-1})$ are ``big'' in the sense that they all have finite GK-codimension. In addition, \cite{PhamTimmins2025} constructed abundant principal ideals of $\Ua(\W{-1})$ that are not finitely generated as one-sided ideals. Our paper vastly strengthens the evidence for \Cref{conj: UW ACC} by showing that a family of $\Ua(\W{-1})$ images of GK-dimension $n$ - each ``approximating'' $\Ua(\W{-1})$ better as $n$ increases - all satisfy ACC on two-sided ideals.

\begin{theorem}[{\Cref{theo: Bn Cn ACC}}]\label{theo: ACC}
    For all $n$, the image $B_n$ of $n$th orbit homomorphism defined in \eqref{eq: Psi map} satisfies ACC on two-sided ideals.
\end{theorem}

Let us explain the motivation and construction of the orbit homomorphisms. Since classifying all irreducible representations of the Witt or the Virasoro algebra is intractable, we followed Dixmier's philosophy and instead characterized primitive ideals of their enveloping algebra, i.e., annihilators of simple modules. When $\g$ is a finite-dimensional solvable Lie algebra, there exists a homeomorphism, named the \emph{Dixmier map}, between the space of coadjoint orbits of $\g^*$ and the space of primitive ideals of $\Ua(\g)$ \cite[Theorem 6.5.12]{dixmier1996enveloping}. Although the map might not be well-defined for a finite-dimensional simple Lie algebra, a version of the Dixmier map for the Witt and Virasoro algebra was successfully constructed in \cite[Theorem 1.1]{pham2025orbit}.

Central to the constructions and results in \cite{pham2025orbit} are the family of algebra maps, called \emph{orbit homomorphisms} \cite[Theorem 1.4]{pham2025orbit}. The one-sided Witt algebra $\W{-1} = \Der(\CC[t])$ has a basis $e_i = t^{i+1}\del$ for $i \geq -1$, where $\del = \frac{d}{dt}$. The Lie subalgebras $\W{n} = \CC \{e_{i} \mid i \geq n\}$ are Lie ideals of $\W{0}$ and thus induce finite-dimensional solvable subquotients $\g_n = \W{0}/\W{n}$ with basis elements $v_i = e_i + \W{n}$. Let $A_1$ be the first Weyl algebra generated by $t$ and $\del$ with $\del t - t \del = 1$. The $n$th \emph{orbit homomorphism} is the following algebra homomorphism \cite[Theorem 1.5]{pham2025orbit}
\begin{equation}\label{eq: Psi map}
\Psi_n : \Ua(\W{-1}) \to A_1 \otimes \Ua(\g_n), \quad 
 f\del \mapsto f(t)\del + \sum_{i=0}^n \frac{f^{(i+1)}(t)}{(i+1)!} v_i.
\end{equation} 
We note that $\Psi_0$ is the natural map $\Ua(\W{-1}) \to A_1$, and the homomorphism used in \cite{sierra2016maps} is in fact $\Psi_1$.
Furthermore, in \cite{pham2025orbit}, these orbit homomorphisms are used to ``lift'' the Dixmier map from finite-dimensional solvable Lie algebra $\g_n$ to the infinite-dimensional settings of $\W{-1}$ and $W$. Thus, understanding the images of these homomorphisms will reveal interesting properties of $\Ua(\W{-1})$. Let $T_n = A_1 \otimes \Ua(\g_n)$ and $B_n = \Psi_n(\Ua(\W{-1}))$. In this paper, we completely classify the ideals and spectra of $B_n$.

As $B_n$ approximates $\Ua(\W{-1})$ as $n$ grows larger, \Cref{theo: ACC} could be seen as a significant step towards \Cref{conj: UW ACC}. We note that the orbit homomorphisms can be defined and \Cref{theo: ACC} also holds for the full Witt algebra $W = \Der(\CC[t, t^{-1}])$ (see \Cref{theo: Bn Cn ACC}).

Interestingly, by using the ideal structure of $B_n$ and noncommutative localization, we completely characterize the prime spectrum of $B_n$. In particular, we show that all but one prime ideals of $B_n$ are induced from $T_n$ and in fact show that restriction gives an ``almost homeomorphism'' between $\Spec B_n$ and $\Spec T_n$ (both equipped with the Jacobson topology).
\begin{theorem}[\Cref{theo: Spec almost homeo}]\label{theo: Spec almost homeo intro}
    For all $n$, the map 
    \[\mathrm{Res}_n: \Spec T_n \to \Spec B_n, \quad P \mapsto P \cap B_n,\]
    gives a continuous, open, near-bijective homomorphism.
\end{theorem}

As a corollary, we see that all prime ideals of $B_n$ are completely prime.
\begin{restatable}{corollary}{prcprime}\label{prime cprime}
    All prime ideals of $B_n$ are completely prime.
\end{restatable}

All known prime ideals of $B_n$ are completely prime. \Cref{prime cprime} provides significant evidence towards the following conjecture.
\begin{conjecture}[{\cite[Conjecture 8.7]{pham2025orbit}}]
    All prime ideals of $\Ua(\W{-1})$ are completely prime.
\end{conjecture}

Since all prime ideals $P$ of $B_n$ are completely prime, we can form the ring of fractions $Q(B_n/P)$ \cite[Theorem 4.15]{krause2000growth} and consider its extended center $Z(Q(B_n/P))$. 
We then completely classify primitive ideals of $B_n$ by showing that $B_n$ satisfies the \emph{Dixmier-Moeglin equivalence} (DME), which characterizes three related properties of a prime ideal $P$ of $B_n$ with different flavours: topological, representation-theoretic, and algebraic.
\begin{restatable}{theorem}{BnDME}\label{theo: BnDME}
    For all $n$, the image $B_n$ satisfies the DME. In particular, the following are equivalent for a prime ideal of $B_n$:
    \begin{enumerate}
        \item[(LC)] $P$ is locally closed, i.e. $P \subsetneqq \bigcap \{ Q \in \Spec B_n \mid Q \supsetneqq P\}$;
        \item[(P)] $P$ is primitive;
        \item[(R)] $P$ is rational, i.e. $Z(Q(B_n/P)) = \CC$.
    \end{enumerate}
\end{restatable}

Although the zero ideal $0$ provides the counter-example to the conjecture that $\Ua(\W{-1})$ satisfies the DME ($0$ is rational but not locally closed \cite[Proposition 3.6.8]{phamthesis}), \Cref{theo: BnDME} suggests that $0$ might be the only prime ideal of $\Ua(\W{-1})$ where the DME fails, as it does not arise from the orbit method in \cite[Theorem 1.1]{pham2025orbit} and \cite{Ushirobira1998}.

\begin{conjecture}
    The enveloping algebra $\Ua(\W{-1})$ satisfies the DME on nonzero prime ideals.
\end{conjecture}

As a result of \Cref{theo: BnDME}, we show that the map $\mathrm{Res}: \Spec T_n \to \Spec B_n$ in \Cref{theo: Spec almost homeo intro} induces an ``almost homeomorphism'' between $\Prim T_n$ and $\Prim B_n$, providing a complete characterization of primitive ideals of $B_n$. 

\begin{theorem}\label{theo: SpecalmosthomeoPrim intro}
    For all $n$, the map $\mathrm{Res}_n: \Spec T_n \to \Spec B_n$ in \Cref{theo: Spec almost homeo intro} gives a well-defined continuous, open and near-bijective map 
    \begin{equation}
        \mathrm{Res}_n: \Prim T_n \to \Prim B_n.
    \end{equation}
\end{theorem}

Our results have several applications to the properties of the Dixmier map for the Witt algebra constructed in \cite{pham2025orbit}. The classification of primitive ideals of $B_n$ in \Cref{theo: BnDME} makes significant progress towards the injectivity, surjectivity and topological properties of this map. In particular, we completely characterize the fiber of the Dixmier map for the fundamental building blocks known as \emph{one-point local functions}, which we will define in Section \ref{subsec: app on Dix}. In that setting, the Dixmier map for $\W{-1}$ is injective except for a single case, whose fibre is exactly two pseudo-orbits \Cref{prop: Dx inject}. We also prove in \Cref{prop: inverse inclusion} that the Dixmier map is an order isomorphism(it preserves and reflects inclusion) for these local functions. Furthermore, since all of primitive ideals of $B_n$ are in the image of the Dixmier map for $\W{-1}$, we make significant progress to the following conjecture.
\begin{conjecture}[{\cite[Conjecture 8.6]{pham2025orbit}}]
    The Dixmier map for $\W{-1}$ is surjective.
\end{conjecture}

\vspace*{15pt}

The organization of this paper is as follows. In Section \ref{sec:preliminaries}, we recall the construction of the orbit homomorphisms, discuss Poisson algebra and Poisson modules and review the Dixmier-Moeglin equivalence in the context of non-Noetherian rings. Then in Section \ref{sec: ACC}, we prove \Cref{theo: ACC} that the image $B_n$ of the $n$th orbit homomorphism satisfies ACC on two-sided ideals by showing that its associated graded $\barB_n$ with respect to the canonical filtration satisfies ACC on graded Poisson ideals. In Section \ref{sec: prime Bn}, we characterize prime ideals of $B_n$ using a combination of localization and considering homomorphic images. In particular, we prove \Cref{theo: Spec almost homeo intro} and \Cref{prime cprime} there. Finally, in Section \ref{sec: DME and Prim Bn}, we prove \Cref{theo: BnDME} that $B_n$ satisfies the Dixmier-Moeglin equivalence, and completely characterize the primitive ideals of $B_n$ by restricting primitive ideals from $T_n$ (\Cref{theo: SpecalmosthomeoPrim intro}). We then give several applications of these results on the properties of the Dixmier map for $\W{-1}$.

\textit{Acknowledgements.}  This work was done as part of the author's PhD research at the University of Edinburgh. A version of this paper was also included as a final chapter in the author's doctoral thesis \cite{phamthesis}. I would like to heartily thank my supervisor, Susan J. Sierra for  her invaluable guidance, very interesting discussions, and numerous suggestions. The author also warmly thanks Stephane Launois and James Timmins for their helpful comments. 

\section{Preliminaries} \label{sec:preliminaries}
Throughout this paper, assume $\kk$ is an uncountable algebraically closed field of characteristic 0. All tensor products and vector spaces will be over $\kk$ unless otherwise specified, so we abbreviate the notation $\otimes_\kk$ to $\otimes$.

We first recall the construction and some notation for the orbit homomorphisms in \cite[Theorem 1.4]{pham2025orbit}. Then in Subsection \ref{subsec: Poisson module background}, we review Poisson modules for a Poisson algebra $A$, which will be the main ingredient for Section \ref{sec: ACC}. Finally, we discuss the Dixmier-Moeglin equivalence for non-Noetherian algebras. 

\subsection{The orbit homomorphism for \texorpdfstring{$\W{-1}$}{W1} and \texorpdfstring{$W$}{W}}
In this subsection, we recall the Lie algebras of interest, a canonical filtration on its enveloping algebra and the orbit homomorphisms defined in \cite[Theorem 1.4]{pham2025orbit}. 

\begin{definition}
    Let $W = \Der (\kk[t, t^{-1}]) = \kk[t, t^{-1}]\del$ be the \emph{(full) Witt algebra}, where $\del = \frac{d}{dt}$. Explicitly, it has the following Lie bracket
    \[ [f\del, g\del] = (fg'-f'g)\del \quad \text{for all} \quad f, g \in \kk[t, t^{-1}].\]
    The Witt algebra is also known as the Lie algebra of vector field on $\kk^\times$ and has a countable basis $e_i = t^{i+1}\del$ for $i \in \ZZ$.
\end{definition}

The Witt algebra has a unique nontrivial central extension, known as the \emph{Virasoro algebra}, that is essential in mathematical physics, string theory and vertex operator algebra. 

\begin{definition}
    For all $n \geq -1$, we define $\W{n} = \{e_i \mid i \geq n\}$, which is a Lie subalgebra of $W$. The \emph{one-sided Witt algebra} $\Der (\kk[t]) = \kk[t]\del$ is in fact the subalgebra $\W{-1}$ of $W$. It is also known as the Lie algebra of vector field on $\kk$.

    As $\W{n}$ is a Lie ideal of $\W{0}$, we can form the $n$th-dimensional quotient $\W{0}/\W{n}$ and denote it by $\g_n$. Let $v_i = e_i + \W{n} \in \g_n$ for $0 \leq i \leq n-1$, which form a basis for $\g_n$.
\end{definition}

We are now ready to define the codomain of the orbit homomorphisms
\begin{definition}
    Let $A_1$ be the \emph{first Weyl algebra} generated by $t, \del$ with the relation $\del t - t \del = 1$. We define 
    \[T_n = A_1 \otimes \Ua(\g_n) \quad \text{and}\quad T_\infty := A_1 \otimes \Ua(\W{0}).\]
    Let $\widetilde{A}_1$ and $\widetilde{T}_n$ be the localization of $A_1$ and $T_n$ at $t$ respectively. By \cite[Corollary 1.2]{PhamTimmins2025}, we note that $T_n$ and $\widetilde{T}_n$ have \emph{Gelfand-Kirillov}(GK)-\emph{dimension} $n+2$. 
\end{definition}

For any Lie algebra $\g$, we can form a unique aossicative algebra having the same category of representations, known as the \emph{universal enveloping algebra} $\Ua(\g)$. We define a natural filtration on the rings of interest $\Ua(\g)$, $T_\infty$, $T_n$ and $\widetilde{T}_n$.

\begin{definition}\label{def: canonical filtration}
    For any Lie algebra $\g$, the enveloping algebra $\Ua(\g)$ has a canonical ascending filtration defined by
    \begin{equation}\label{eq: order filtration}
        \Ua_m(\g) = \sum_i^m \g^i = \kk\{g_1 \dots g_j \mid j \leq m, g_i \in \g\}.
    \end{equation}
    We say that an element of $\Ua_m(\g)\setminus \Ua_{m-1}(\g)$ has \emph{order $m$}. This induces a ascending filtration on $T_\infty$ with mth filtered piece given by
    \[\{\kk[t]{\del^{m_\del}e_0^{m_0} \dots e_N^{m_N} \mid m_\del + m_0 + \dots + m_N \leq m}\},\]
    for all $m \geq 0$. Furthermore we give $T_n$ and $\widetilde{T}_n$ the filtration induced from $T_\infty$. We refer to these filtration as \emph{canonical filtration} on $\Ua(\g), T_\infty$, $T_n$ and $\widetilde{T}_n$.
\end{definition}

We can now restate the following theorem from \cite{pham2025orbit}, which shows that the orbit homomorphisms $\Psi^{\W{-1}}_\infty, \Psi^{\W{-1}}_n, \Psi^{W}_n$ are indeed algebra homomorphisms. 
\begin{proposition}[{\cite[Theorem 1.4]{pham2025orbit}}]
    The following maps are filtered algebra homomorphisms
\begin{align*}
        \Psi^{\W{-1}}_{\infty} : \Ua(W{-1}) \rightarrow {T}_\infty = A_1\otimes \Ua(\W{0}), \quad 
        &f \partial \mapsto f(t)\partial + \sum_{i = 0}^{\infty} \frac{f^{(i+1)}(t)}{(i+1)!}e_i, \\
        \Psi^{\W{-1}}_{n} : \Ua(W{-1}) \rightarrow {T}_n = A_1\otimes \Ua(\g_n), \quad 
        &f \partial \mapsto f(t)\partial + \sum_{i = 0}^{n-1} \frac{f^{(i+1)}(t)}{(i+1)!}v_i, \\
        \Psi^{W}_{n} : \Ua(W) \rightarrow \widetilde{T}_n = \tilde{A}_1 \otimes \Ua(\g_n), \quad 
        &f \partial \mapsto f(t)\partial + \sum_{i = 0}^{n-1} \frac{f^{(i+1)}(t)}{(i+1)!}v_i,
    \end{align*}
\end{proposition}

In this paper, we mainly focus on the image of $\Psi^{\W{-1}}_{n}$, so we abbreviate $\Psi^{\W{-1}}_{n}$ to $\Psi_{n}$, and refer to $\Psi^{W}_{n}$ whenever we discuss the full Witt algebra. Let $B_n = \Psi_n(\Ua(\W{-1}))$ and let $C_n = \Psi^W_n(\Ua(W))$. The aim of this paper is to analyze the two-sided ideal structure of $B_n$ and $C_n$, and in particular their prime and primitive spectrum. 

For $n \geq m$, we have the projection $\tau_{n, m}: \Ua(\g_n) \surj \Ua(\g_m)$ induced by $v_{\geq m} \to 0$. Abusing notation, we let $\Pi_{n, m} = \id \otimes \tau_{n, m}$ denote the both surjection $T_n \surj T_m$ and $B_n \surj B_m$. Finally, we state one of the main results of \cite{PhamTimmins2025} that the image $B_n$ is big in the following sense.
\begin{lemma}\label{lem: big image}
    For all $n \geq 2$, we have $T_n v_{n-1}^2 \subseteq B_n$ and $\widetilde{T}_n v_{n-1}^2 \subseteq C_n$.
\end{lemma}
\begin{proof}
    The result for $B_n$ is \cite[Theorem 1.1]{PhamTimmins2025} and the details for $C_n$ is in \cite[Theorem 4.3.16]{phamthesis}.
\end{proof}

\subsection{Poisson modules}\label{subsec: Poisson module background}
In this subsection, we recall central notions in Poisson algebra and construct the category of Poisson $A$-modules. This will play an important role in the proof of \Cref{theo: ACC}. 

Let $A$ be a commutative $\kk$-algebra and $\{\cdot\ , \cdot\}: A \times A \to A $ a skew-symmetric $\kk$-bilinear map. 
We say that $(A, \{\cdot\ , \cdot\})$ is a {\it Poisson algebra} if $\{\cdot\ , \cdot\}$ satisfies the Leibniz rule on each input and the Jacobi identity. An ideal $I$ of $A$ is called \emph{Poisson} if $\{I, A\}\subseteq I$. For any Lie algebra $\g$, it is well known that the symmetric algebra $\Sa(\g)$ is the associated graded of $\Ua(\g)$ with respect to the filtration \eqref{eq: order filtration}. It also possesses a canonical Poisson algebra structure, induced by defining $\{u, v\} = [u, v]$ for any $u, v \in \g$.

In the literature, there are two definitions of a Poisson $A$-module $M$. We follow the definition introduced in \cite{Oh99Poisson, Farkas2000Poissonmodules}, which we will restate below. The other notion, which is called \emph{weak Poisson action} in \cite{Farkas2000Poissonmodules}, is considered by Brown and Gordon in \cite{brown2002Poisson}, where only \eqref{eq: derivation axiom} needs to hold for the bilinear map $\{\cdot\ , \cdot\}_M$.
\begin{definition}\label{def: Poisson modules}
    Let $(A, \{\cdot\ , \cdot\}_A)$ be a commutative Poisson algebra. An $A$-module $M$ is called a \emph{Poisson $A$-module} if there is a bilinear map $\{\cdot\ , \cdot\}_M: A \times M \to M$ such that for all $a, b \in A$ and all $m \in M$, the following conditions hold:
    \begin{align}
        \{\{a, b\}_A, m\}_M &= \{a, \{b, m\}_M\}_M - \{b, \{a, m\}_M\}_M; \label{eq: Lie algebra axiom}\\ 
        \{a, bm\}_M &= \{a, b\}_A m + b\{a, m\}_M; \label{eq: derivation axiom}\\ 
        \{ab, m\}_M &= a\{b, m\}_M + b\{a, m\}_M. \label{eq: derivation axiom 2}
    \end{align}
\end{definition}

We now explain the meaning and motivation for the conditions in \Cref{def: Poisson modules}. Equation \eqref{eq: Lie algebra axiom} says that $M$ is a module over $A$ when $A$ is viewed as a Lie algebra. Equation \eqref{eq: derivation axiom 2} requires that $\{-, m\}_M$ is a derivation from $A$ to the $A$-module $M$.

To explain Equation \eqref{eq: derivation axiom} more precisely, consider a derivation $\delta$ of a commutative ring $A$. For an $A$-module $M$, we can define a \emph{$\delta$-derivation} of $M$ to be a linear map $\xi : M \to M$ such that
\begin{align*}
    \xi(bm) = \delta(b)m + b\xi(m) \quad \text{for all} \quad b \in A, m \in M. 
\end{align*}
Equation \eqref{eq: derivation axiom} then requires that $\{a, -\}_M$ is a derivation of $M$.

\begin{definition}
    Let $A$ be a Poisson algebra and $M$ a Poisson $A$-module. We call $M'$ a \emph{Poisson $A$-submodule} of $M$ if $\{a, m'\}_M \in M'$ for all $a \in A, m' \in M'$.

    A non-zero Poisson $A$-module $M$ is called \emph{simple} if the only Poisson $A$-submodules of $M$ are $0$ and itself. 
\end{definition}
\begin{definition}
    Let $M, M'$ be Poisson $A$-modules. A \emph{Poisson module homomorphism} $f : M \to M'$ is an $A$-module homomorphism such that 
    \begin{align*}
        f(\{a, m\}_M) = \{a, f(m)\}_{M'} \quad \text{for all} \quad a \in A, m \in M.
    \end{align*}
    If such a map $f$ is bijective then $f^{-1}: M' \to M$ is also a Poisson module homomorphism and $f$ is a Poisson $A$-module isomorphism.
\end{definition}
We have thus formed the category of Poisson $A$-modules.

\subsection{The Dixmier-Moeglin equivalence for non-Noetherian \texorpdfstring{$\kk$}{k}-algebras}

In this subsection, we discuss the Dixmier-Moeglin equivalence (DME) in the context of non-Noetherian noncommutative rings. The following definitions follow \cite[Chapter 9.1]{mcconnell2001noncommutative}, where we restate here for completeness. We first have a ring theoretic definition.
\begin{definition}
    \begin{enumerate}
        \item A ring $R$ has the \emph{radical property} if the Jacobson radical of each factor ring of $R$ is nil. 
        \item A ring $R$ is called \emph{Jacobson} if every prime ideal is an intersection of primitive ideals. In other words, every prime factor ring has zero Jacobson radical.
    \end{enumerate}
\end{definition}

By \cite[Lemma 1.2]{mcconnell2001noncommutative}, a Jacobson ring always satisfies the radical property and moreover, the two definitions are equivalent for rings whose prime factor rings are all right Goldie. We will see that this is exactly the case for $B_n$ in \Cref{lem: Jacobson rings 1}.

The next definition has a more representation theory flavor. Note that for a commutative ring $R$, primitive ideals are maximal. Then the \emph{classical Nullstellensatz} for a finitely generated commutative $\kk$-algebra $R$ states that if $P$ is a maximal ideal of $R$, then $R/P$ is a finite extension of $\kk$. Thus the following definition could be seen as a generalization of the classical Nullstellensatz to noncommutative rings. 

\begin{definition}
    An algebra $R$ over $\kk$ has the \emph{endomorphism property over $\kk$} if, for all simple modules $M_R$, $\End_{R} M$ is algebraic over $\kk$. 

    If $R$ satisfies the endomorphism property over $\kk$ and has the radical property then $R$ satisfies \emph{the Nullstellensatz over $\kk$}.
\end{definition}
Our assumption on the cardinality of $\kk$, implies that all countably generated $\kk$-algebra satisfies the the Nullstellensatz over $\kk$ \cite[Corollary 9.1.8]{mcconnell2001noncommutative}. The following lemma allows us to discuss freely the endomorphism ring of a simple module $M$.
\begin{lemma}\label{eq: End M}
    Let $M$ be a left module of $R$ with annihilator $P$. Then 
    \begin{equation*}
        \End_R M \cong \End_{R/P} M,
    \end{equation*}
    as $\kk$-algebras. \qed
\end{lemma}

We now recall three interesting and related properties of prime ideals in $R$ with different lenses: topological, representation-theoretic, and algebraic.
\begin{definition}
    A prime ideal $P$ or $R$ is called 
    \begin{itemize}
        \item[(LC)] \emph{locally closed} if $\{P\}$ is locally closed in $\Spec R$ under the Zariski topology, i.e., 
        \[P \subsetneqq \bigcap \{ Q \in \Spec B_n \mid Q \supsetneqq P\};\]
        \item[(P)] \emph{primitive}  if $P$ is an annihilator of a (left) $R$-module;
        \item[(R)] \emph{rational} if the center $Z(Q(R/P))$ of the ring of fractions $Q(R/P)$ equals $\kk$.
    \end{itemize}
    A ring $R$ is called to satisfy the \emph{Dixmier Moeglin equivalence} (DME) if the above conditions are equivalent for all prime ideals $P$ of $B_n$, i.e., 
    \begin{align*}
        \text{(LC)} \Leftrightarrow \text{(P)} \Leftrightarrow \text{(R)}.
    \end{align*} 
\end{definition}
The equivalence is named after Dixmier and Moeglin, who proved in their seminal papers \cite{dixmier1977ideaux} and \cite{Moeglin1980ideaux} that the enveloping algebra of a finite-dimensional Lie algebra satisfies the DME.

The following result in \cite[4.1.6]{dixmier1996enveloping} is important to show that when a ring $R$ satisfies the Nullstellensatz over $\kk$, a primitive ideal $P$ of $R$ is rational. A careful look at the proof of \cite[4.1.6]{dixmier1996enveloping} or \cite[Lemma II.7.13]{BrownGoodearBook} shows that the condition $R$ is (left) Noetherian is not needed, only that $R/P$ has a Goldie ring of quotients.
\begin{lemma}[{\cite[Lemma II.7.13]{BrownGoodearBook}}]\label{lem: Nullstellensatz property}
    Let $R$ be a ring and let $P$ be the annihilator of a simple left $R$-module $M$. Suppose $R/P$ is Goldie, and let $Q(R/P)$ be the quotient ring of $R/P$. Then there exists a $Z(R/P)$-algebra embedding 
    \begin{align*}
        Z(Q(R/P)) \to Z(\End_R M).
    \end{align*}
    In particular, if $R$ is a $\kk$-algebra, then the map above is a $\kk$-algebra embedding.
\end{lemma}
\begin{proof}
    By \Cref{eq: End M}, we can assume $P = 0$. Let $Q = Q(R)$, $E = \End_{R} M$ and let
    \begin{align}
        \mathcal{Z} = \{(r, s) \in R \times R \mid s \text{ is regular and } s^{-1}r \in Z(Q)\}.
    \end{align}
    Note that if $(r, s) \in \mathcal{Z}$, then for all $a \in R$
    \begin{align*}
        sar = sas(s^{-1}r) = s(s^{-1}r) as = ras,
    \end{align*}
    so in particular, $sr=rs$.

    We next claim that for any $(r, s) \in \mathcal{Z}$, there exists a unique $f \in E$ such that $f(sm) = rm$ for all $m \in M$. Since $M$ is faithful, we may choose $m_0 \in M$ such that $sm_0 \neq 0$. Then simplicity of $M$ implies $sm_0$ generates $M$, and so there can be at most one such $f \in E$ satisfying $f(sm_0) = rm_0$. We now show that $\Ann_R(sm_0) \subseteq \Ann_R(rm_0)$. Suppose the contrary, there exists $a \in R$ with $asm_0 = 0$ but $arm_0 \neq 0$. Then $arm_0$ generates $M$ and so $barm_0 = m_0$ for some $b \in R$. But then $sm_0 = s(ba)rm_0 = r(ba)sm_0 = 0$, a contradiction. Thus $\Ann_R(sm_0) \subseteq \Ann_R(rm_0)$, and since $sm_0$ generates $M$, there exists a unique $f \in E$ such that $f(sm_0)=rm_0$. Now any $m \in M$ has the form $m = csm_0$ for some $c \in R$, so 
    \begin{align*}
        f(sm) = sf(csm_0) = scf(sm_0) = scrm_0 = rcsm_0 = rm.
    \end{align*}

    Further, for any $g \in E$, we have 
    \begin{align*}
        gf(sm_0) = g(rm_0) = rg(m_0) = f(sg(m_0)) = fg(sm_0),
    \end{align*}
    and so $gf = fg$, i.e. $f \in Z(E)$. 
    
    Next, consider $(r_1, s_1)$ and $(r_2, s_2)$ in $\mathcal{Z}$ such that $s^{-1}_1 r_1 = s^{-1}_2 r_2$, and let $f_1, f_2 \in Z(E)$ correspond to these pairs as in the claim. Since $r_1$ and $s_1$ commute, $s_2 r_1  =  r_2 s_1$. Hence for all $m \in M$,
    \begin{align*}
        f_1(s_2s_1 m) = s_2 f_1(s_1 m) = s_2 r_1 m = r_2 s_1 m = f_2 (s_2 s_1 m).
    \end{align*}
    Since $s_2s_1 \neq 0$, there exists some $m' \in M$ such that $s_2s_1m' \neq 0$, so $f_1 = f_2$.

    Thus, we have a well-defined map $\theta: Z(Q) \to Z(E)$ such that $\theta(s^{-1}r)(sm) = rm$ for all $(r, s) \in \mathcal{Z}(R)$. As $Z(Q)$ is a field, $\theta$ must be an embedding.
\end{proof}

We now show that for many well-behaved rings, the classical implications
\begin{align*}
    \text{(LC)} \Rightarrow \text{(P)}\Rightarrow \text{(R)}, 
\end{align*}
still hold even without a Noetherian hypothesis. 

\begin{proposition}\label{prop: classical DME direction}
    Let $R$ be a Jacobson ring that satisfies the Nullstellensatz over $\kk$. Moreover, assume that $R/P$ is Goldie for all prime ideals $P$ of $R$. Then 
    \begin{align*}
        \text{(LC)} \Rightarrow \text{(P)} \Rightarrow \text{(R)}.
    \end{align*}
\end{proposition}
\begin{proof}
    Let $P$ be a locally closed prime ideal of $R$. Then it is well-known that $P$ is primitive if and only if the Jacobson radical $J(R/P) = 0$ (see, e.g. \cite[Proposition 2.2]{BellDME2019}). Thus, since $R$ is Jacobson, $P$ is primitive. 

    Let $P$ be an annihilator of a simple module $M$. By \Cref{lem: Nullstellensatz property}, we have a $\kk$-algebra embedding $Z(Q(R/P)) \hookrightarrow \End_{R} M$. Since $R$ satisfies the Nullstellensatz over $\kk$, $\End_{R} M$ is algebraic over $\kk$, so $Z(Q(R/P))$ is algebraic over $\kk$. 
\end{proof}

\section{The ascending chain condition on ideals of \texorpdfstring{$B_n$}{Bn} and \texorpdfstring{$\widetilde{B}_n$}{Bn}}\label{sec: ACC}

The main goal of this section is to show that $B_n$ and $C_n$ both satisfy the ascending chain condition (ACC) on two-sided ideals. 

\begin{theorem}\label{theo: Bn Cn ACC}
    For all $n$, $T_n$ is a Noetherian $(B_n, B_n)$-bimodule and $\widetilde{T}_n$ is a Noetherian $(C_n, C_n)$-bimodule. As a result, $B_n$ and $C_n$ satisfy ACC on two-sided ideals.
\end{theorem}

The proof of \Cref{theo: Bn Cn ACC} utilizes the associated graded rings $\barT_n$ and $\barB_n$ of $T_n$ and $B_n$ with respect to the canonical filtration defined in \Cref{def: canonical filtration}. Further understanding of Poisson modules of $\barB_n$ will be crucial to \Cref{theo: ACC} and the key step is the following result.
\begin{theorem}[{\Cref{prop: barBn ACC graded stronger}}]\label{theo: ACC graded}
    For all $n$, $\barT_n$ is a graded Poisson Noetherian $\barB_n$-module and $\barB_n$ satisfies ACC on graded Poisson ideals. The similar statement holds for $\barC_n$.
\end{theorem}

Firstly, for the case $n=0$, \Cref{theo: Bn Cn ACC} can be proven directly since the image $B_0$ is the idealizer of a maximal left ideal in $A_1$. 
\begin{proposition}\label{prop:B0noe}
    The image $B_0$ is left and right Noetherian, thus it satisfies ACC on two-sided ideals.
\end{proposition}
\begin{proof} 
    It is easy to see that the image $B_0 = \kk \oplus A_1 \del$. Then $B_0$ is the idealizer in $A_1$ of the maximal left ideal $A_1 \del$. Thus, from \cite{Robson1972}, $B_0$ is left and right Noetherian. 
\end{proof}

In fact, the ideal structures of $B_0$ and $C_0$ are particularly easy to understand.
\begin{proposition}\label{prop:B0ideals}
    The only proper ideals of $B_0$ are $0$ and $A_1\del$, which are both completely prime ideals of $B_0$. The only proper ideals of $C_0$ are $0$ and $\widetilde{A}_1\del$, which are both completely prime ideals of $C_0$. 
\end{proposition}
\begin{proof}
    The proofs for $B_0$ and $C_0$ are the same, so we show for $B_0$. Let $I$ be a proper nonzero ideal of $B_0$. Then $A_1 \del I A_1$ is a nonzero ideal of $A_1$. Since $A_1$ is a simple ring, $A_1 \del I A_1 = A_1$. As a result, 
    \[A_1 \del=(A_1 \del) I (A_1 \del) \subseteq B_0 I B_0 \subseteq I.\] 
    Thus, $I=A_1 \del$. Since $B_0/ A_1 \del = \kk$ is a domain, $A_1 \del$ is completely prime.
\end{proof}

The canonical order filtration in \Cref{def: canonical filtration} will play an important role in proving that $B_n$ and $C_n$ satisfy the ACC on two-sided ideals. We now give the notation for their associated graded.
\begin{notation}\label{def:filt}
    Let $\mc{F}$ be the canonical filtration on $T_n$ and $\widetilde{T}_n$ in \Cref{def: canonical filtration} 
    and we denote $\gr(t) = \bart, \gr(\del) = y$ and $\gr(v_j) = \barv_j$. We denote their associated graded rings by
    \begin{align*}
        \barT_n = \gr_\mc{F} T_n = \bigoplus_{i \in \NN} \mc{F}^i T_n/ \mc{F}^{i+1} T_n, \quad  \barT_n[\bar{t}^{-1}] = \gr_\mc{F} \widetilde{T}_n = \bigoplus_{i \in \NN} \mc{F}^i \widetilde{T}_n/ \mc{F}^{i+1} \widetilde{T}_n.
    \end{align*}
    Moreover, $\mc{F}$ induces a subset filtration on any $S \subseteq T_n$ and we denote $\bar{S} = \gr_\mc{F}(S)$. Since $\Psi_n$ and $\Psi^W_n$ are filtered homomorphisms by \cite[Theorem 1.4]{pham2025orbit}, the subset filtrations on $B_n$ and $C_n$ coincide with the induced canonical filtration from $\UW$ and $\Ua(W)$. Let 
    \begin{align*}
        \barB_n = \gr_\mc{F} B_n, \quad  \barC_n = \gr_\mc{F} C_n.
    \end{align*}
\end{notation}
The following lemma follows immediately from standard arguments. The Poisson structure on $\barT_n$ and $\barT_n[\bar{t}^{-1}]$ both come from the associated graded construction: if $a \in \mc{F}^{i} T_n \setminus \mc{F}^{i-1} T_n$ and $b \in \mc{F}^{j} T_n \setminus \mc{F}^{j-1} T_n$, then 
\begin{align*}
    \{\bar{a}, \bar{b}\} = ab-ba + \mc{F}^{i+j-1} T_n.
\end{align*}
\begin{lemma}\label{lem:ass gra}
    The associated graded rings $\barT_n$ and $\barT_n[\bar{t}^{-1}]$ are isomorphic to the graded commutative rings 
    \[\barT_n \cong \kk[\bar{t}, y] \otimes \Sa(\g_n) \quad \text{and} \quad \barT_n[\bar{t}^{-1}] \cong \kk[\bar{t}, \bar{t}^{-1}, y] \otimes \Sa(\g_n) \]
    with the grading 
    \begin{equation}\label{eq: grading}
        \deg \bar{t}= 0, \quad \deg y = 1, \quad \deg \bar{v}_i = 1 \quad \text{for all} \quad i \geq 0.
    \end{equation}
    Moreover, $\barT_n$ and $\barT_n[t^{-1}]$ are Poisson algebras, where the only nonzero Poisson brackets of the generators are
    \begin{equation*}\label{eq: Poisson structure}
        \{y, \bar{t}\} = 1, \quad \{y, \bar{t}^{-1}\} = -\bar{t}^{-2}, \quad \{\bar{v}_i, \bar{v}_j\} = (j-i)\bar{v}_{i+j} \quad \text{for all} \quad i, j \geq 0,
    \end{equation*}
    where $\bar{v}_{\geq n} = 0$. \qed
\end{lemma}
Note that although by \Cref{prop:B0noe}, $B_0$ is left and right Noetherian, its associated graded ring $\barB_0$ is not, as shown in the following lemma. 
\begin{lemma}
    The commutative rings $\barB_0$ and $\barC_0$ are not Noetherian.
\end{lemma}
\begin{proof}
    The proofs for $\barB_0$ and $\barC_0$ are the same, so we show for $\barB_0$. With the filtration in \Cref{def:filt}, we have 
    \[\barT_0 = \kk[\bar{t}, y], \quad \barB_0 = \kk + \barT_0 y.\]
    Note that $\barB_0$ is not Noetherian as we have the following ascending chain of ideals
    \[(\bar{t}y) \subseteq (\bar{t}y, \bar{t}^2y) \subseteq  (\bar{t}y, \bar{t}^2y, \bar{t}^3y) \subseteq \dots,  \]
    where each inclusion is strict. Suppose the contrary that 
    \begin{equation}\label{eq: to satisfy}
        \bar{t}^{n+1}y=\sum^{n}_{i=1}f_i  \bar{t}^{i} y, \quad \text{with} \quad f_0,\dots,f_n \in  \barB_0. 
    \end{equation}
    As $\barB_0 = \kk + \barT_0 y$, we write each $f_i$ as a sum of constant and a multiple of $y$, which forces \eqref{eq: to satisfy} to holds for $f_i \in \kk$. This gives a contradiction. 
\end{proof}

However, we will show that $\barB_0$ satisfies the ACC on Poisson ideals and moreover $\barT_0$ is a \emph{graded Poisson Noetherian $\barB_0$-module}, a central notion which we will make precise now. Recall the definition of Poisson modules in \Cref{def: Poisson modules}. 

\begin{definition}
    Let $A \subseteq B$ be (commutative) Poisson algebras. In that case any weak Poisson $A$-submodule $M$ of $B$ is a Poisson $A$-submodules of $B$ and the two notions coincide, since $\{\cdot\ , \cdot\}_M$ is the restriction of the Poisson bracket of $B$.

    If $A$ is also a graded ring, by a \emph{graded Poisson $A$-module} $M$ we mean a Poisson $A$-module that is at the same time graded $A$-module such that 
    \begin{align*}
        \{G_i A, G_j M\}_M \subseteq G_{i+j-1}M,
    \end{align*}
    where $G_i A$ and $G_j M$ are the $i$ and $j$th graded pieces of $A$ and $M$ respectively. 

    We call $B$ a \emph{(graded) Poisson Noetherian $A$-module} if $B$ satisfies the ascending chain condition (ACC) on (graded) Poisson $A$-submodules.
\end{definition}

\begin{proposition}\label{Poisson B0}
    The only graded Poisson $\barB_0$-submodules of $\barT_0$ are $0$ and $\barT_0 y^m$ for some $m \in \NN$. Hence, $\barT_0$ is a graded Poisson Noetherian $\barB_0$-module, and $\barB_0$ satisfies ACC on graded Poisson ideals.

    Similarly, the only graded Poisson $\barC_0$-submodules of $0$ and $\barT_0[\bar{t}^{-1}]$ are $\barT_0[\bar{t}^{-1}] y^m$ for some $m \in \NN$. Hence, $\barT_0[\bar{t}^{-1}]$ is a graded Poisson Noetherian $\barC_0$-module, and $\barC_0$ satisfies ACC on graded Poisson ideals.
\end{proposition}

\begin{proof}
    To provide the proof for both statements, we work over $\barC_0$ and $\barT_0[\bar{t}^{-1}]$. Let $J$ be a graded Poisson $\barC_0$-submodule of $\barT_0[\bar{t}^{-1}]$. If $0 \neq f \in J$, since $y \bart^a \in \barC_0$ for all $a$, then
    \begin{equation}\label{eq: canceling t-1}
        \{y\bar{t}^a, f \}  = \bar{t}^a \ {d f}/{d \bar{t}} \ + y\{\bar{t}^a, f \} \in J.
    \end{equation}
    Thus choosing big enough $a$, we can assume $f \in \barT_0$. As  
    \begin{equation}\label{eq: lowering power of t}
        {d f}/{d \bar{t}} = \{y, f \} \in J
    \end{equation}
    and $J$ is graded, there exists some $y^m \in J$. Let $m$ be the minimum integer such that $y^m \in J$. Then $mk y^m \bar{t}^{k-1} = \{y^m, \bar{t}^k y\} \in J$ for all $k \geq 2$, thus $J \supseteq \barT_0 y^m$. 
    
    It is easy to see that $J = \barT_0[\bar{t}^{-1}] y^m$, as otherwise if $f(\bar{t})y^{m-1} \in J$, again by \eqref{eq: canceling t-1}, we can assume that $f(\bar{t}) \in \kk[\bar{t}]$. Then by repeatedly applying \eqref{eq: lowering power of t} gives us that $y^{m-1} \in J$, contradicting minimality of $m$. 
\end{proof}

\Cref{Poisson B0} will be the base case for our induction argument. The following result is standard, but we provide the proof for completeness.

\begin{proposition}\label{prop: ACC graded induce ACC }
    If $\barT_n$ is a graded Poisson Noetherian $\barB_n$-module, then $T_n$ is a Noetherian $(B_n, B_n)$-bimodule. The same statement holds for $\barT_n[\bar{t}^{-1}]$ and $\barC_n$. 
\end{proposition}
\begin{proof}
    It is standard that if $I$ is a $(B_n, B_n)$-submodule of $T_n$, then $\gr I$ is a graded Poisson $\barB_n$-submodule of $\barT_n$ (see, e.g., \cite[Example 4]{Farkas2000Poissonmodules} for the Poisson statement). The associated graded map induces the following map from the $(B_n, B_n)$-submodule lattice $\mc{BM}(T_n)$ of $T_n$ to the graded Poisson $\barB_n$-submodule lattice $\mc{GP}(\barT_n)$ of $\barT_n$, which we denoted by $\gr$ as well, 
    \[\gr: \mc{BM}(T_n) \to \mc{GP}(\barT_n), \quad I \mapsto \gr(I) = \bar{I}.\]
    It is enough to show that $\gr$ preserves inclusion and is injective on chains. Let $I \subsetneqq J$ be $(B_n, B_n)$-submodules of $T_n$. Then we claim that $\bar{I} \subsetneqq \bar{J}$. It is clear that $\bar{I} \subseteq\bar{J}$, so we only need to show that they are distinct.
    
    Suppose for the contrary that $\bar{I} =\bar{J}$, then let $a$ be a minimal degree element in $J \backslash I$. Then there exists $b \in I$, such that $\bar{a} = \bar{b} \in \bar{I} =\bar{J}$. Thus $a-b$ is an element of lower degree that lies in $J \backslash I$, contradicting the choice of $a$. 
\end{proof}

To prove \Cref{theo: Bn Cn ACC}, it thus suffices to show that for all $n$, $\barT_n$ is a graded Poisson Noetherian $\barB_n$-module or \Cref{theo: ACC graded}. However, to do so, we need to prove a slightly stronger statement. To that end, let us introduce some more notation. 

\begin{definition}\label{def: Poisson mod Tva}
    For $0 \leq i \leq n-1$, let $\mathbf{a} = (a_{i}, \dots, a_{n-1}) \in \NN^{n-i}$ and let $\barv^{\mathbf{a}} = \barv_{i}^{a_i}\dots \barv_{n-1}^{a_{n-1}}$ and $\mc{A} = \sum_{k=i}^{n-1} k a_k$. We consider the following $\barT_{i}$-module
    \begin{equation}\label{eq: Poisson mod Tva}
        \barT_{i}\barv^{\mathbf{a}} := \barT_{i}\barv_{i}^{a_i}\dots \barv_{n-1}^{a_{n-1}},
    \end{equation}
    where \eqref{eq: Poisson mod Tva} is a formal product since $\barv^{\mathbf{a}} \notin T_i$. Let 
    \begin{equation}\label{eq: Poisson structure 1}
        \begin{split}
            \{v_0, \barv^{\mathbf{a}}\}_{\barT_{i}\barv^{\mathbf{a}}} &= \mc{A} \barv^{\mathbf{a}}, \\ 
            \{\del, \barv^{\mathbf{a}}\}_{\barT_{i}\barv^{\mathbf{a}}} &= \{t, \barv^{\mathbf{a}}\}_{\barT_{i}\barv^{\mathbf{a}}} = \{v_j, \barv^{\mathbf{a}}\}_{\barT_{i}\barv^{\mathbf{a}}} = 0 \quad \text{for all} \quad j \geq 0.
        \end{split}
    \end{equation}
    We then extend \eqref{eq: Poisson structure 1} to a map $\{\cdot , \cdot \}_{\barT_{i}\barv^{\mathbf{a}}}: \barT_{i} \times  \barT_{i}\barv^{\mathbf{a}} \to  \barT_{i}\barv^{\mathbf{a}}$ by \eqref{eq: derivation axiom}, \eqref{eq: derivation axiom 2}, and then bilinearly. Explicitly, for $r_1, r_2, r_3 \in \barT_{i}$, we have 
    \begin{align}
        \{r_1, r_2 r_3 \barv^{\mathbf{a}}\}_{\barT_{i}\barv^{\mathbf{a}}} &= \{r_1, r_2\}_{\barT_{i}} r_3 \barv^{\mathbf{a}} + r_2\{r_1, r_3 \barv^{\mathbf{a}}\}_{\barT_{i}\barv^{\mathbf{a}}}, \label{eq: new derivation axiom 1}\\
        \{r_1r_2, r_3 \barv^{\mathbf{a}}\}_{\barT_{i}\barv^{\mathbf{a}}} &= r_1\{r_2, r_3 \barv^{\mathbf{a}}\}_{\barT_{i}\barv^{\mathbf{a}}} + r_2\{r_1, r_3 \barv^{\mathbf{a}}\}_{\barT_{i}\barv^{\mathbf{a}}} \label{eq: new derivation axiom 2},
    \end{align}
    where $\{\cdot , \cdot \}_{T_i}$ is the Poisson bracket on $T_i$.
\end{definition}

Recall the Lie ideal $\W{1}/\W{i} = \kk \{v_1, \dots, v_{i-1}\}$ of $\g_i = \W{0}/\W{i}$. It is easy to see that 
\begin{equation}\label{eq: Poisson structure not v0}
    \{r, r'\barv^{\mathbf{a}}\}_{\barT_{i}\barv^{\mathbf{a}}} = \{r, r'\}_{\barT_{i}} \barv^{\mathbf{a}} , \quad \text{if} \quad r \in \kk[\bart, y] \otimes \Sa(\W{1}/\W{i}), 
\end{equation}
i.e., if $r$ does not involve $v_0$. We now show that \eqref{eq: Poisson structure 1}, \eqref{eq: new derivation axiom 1} and \eqref{eq: new derivation axiom 2} make $\barT_{i}\barv^{\mathbf{a}}$ a Poisson $\barT_{i}$-module. 
\begin{lemma}
    The bilinear map $\{\cdot , \cdot \}_{\barT_{i}\barv^{\mathbf{a}}}$ in \Cref{def: Poisson mod Tva} makes $\barT_{i}\barv^{\mathbf{a}}$ a Poisson $\barT_i$-module, and thus a Poisson $\barB_i$-module.
\end{lemma}
\begin{proof}
Since $\{\cdot , \cdot \}_{\barT_{i}\barv^{\mathbf{a}}}$ automatically satisfies \eqref{eq: derivation axiom} and \eqref{eq: derivation axiom 2} by construction, it suffices to check \eqref{eq: Lie algebra axiom} for $a=r_1, b = r_2$ and $m = r_3 \barv^{\mathbf{a}}$, where $r_1, r_2, r_3 \in \{\bart, y, \barv_0, \dots, \barv_{i-1}\}$. The only nontrivial case is $r_1 = \barv_0, r_2=\barv_k$ and $r_3 = \barv_j$ for $1 \leq k, j \leq n-1$ and $k+j \leq i-1$. In that case,
\begin{align*}
    &\{\barv_0, \{\barv_k , \barv_j \barv^{\mathbf{a}}\}_{\barT_{i}\barv^{\mathbf{a}}}\}_{\barT_{i}\barv^{\mathbf{a}}} - \{\barv_k, \{\barv_0 , \barv_j \barv^{\mathbf{a}}\}_{\barT_{i}\barv^{\mathbf{a}}}\}_{\barT_{i}\barv^{\mathbf{a}}}\\
 &= \left ((j-k) (\mc{A} + j+k) - (j-k) (\mc{A} +j) \right )\barv_{j+k} \barv^{\mathbf{a}}\\
 &= \{\{\barv_0, \barv_k\} , \barv_j \barv^{\mathbf{a}}\}_{\barT_{i}\barv^{\mathbf{a}}}.
\end{align*}
Thus, the lemma follows.
\end{proof}

As explained before, to prove \Cref{theo: Bn Cn ACC}, it suffices to show that $\barT_i$ is a graded Poisson Noetherian $\barB_i$-module. However, we need to show the following stronger statement; the reason for this will be explained in \Cref{remark: why induction like this}. 
\begin{theorem}\label{prop: barBn ACC graded stronger}
    For all $i, n \in \NN$ satisfying $i \leq n-1$, and for all $\mathbf{a} = (a_{i}, \dots, a_{n-1}) \in \NN^{n-i}$, $\barT_i \barv^{\mathbf{a}}$ is a graded Poisson Noetherian $\barB_i$-module. 
    
    In particular, $\barT_n$ is a graded Poisson Noetherian $\barB_n$-module and $\barB_n$ satisfies ACC on graded Poisson ideals.
\end{theorem}

Recall that the images $B_n$ and $C_n$ are ``big'' in the sense that they contain the principle ideals $T_n v_{n-1}^2$ and $\widetilde{T}_n v_{n-1}^2$ respectively, and are birational to $T_n$ and $\widetilde{T}_n$ respectively (see \cite[Corollary 1.2]{PhamTimmins2025}). The analog for their associated graded rings is as follows.
\begin{lemma}\label{lem:barBn contains}
    For all $i \geq 1$, we have 
    \[\barB_i \supseteq \barT_i \bar{v}_{i-1}^2, \quad \text{and} \quad  \barC_i \supseteq \barT_i[\bar{t}^{-1}] \bar{v}_{i-1}^2.\]
\end{lemma}
\begin{proof}
    The proofs for $\barB_i$ and $\barC_i$ are the same, so we only provide for $\barB_i$ to ease the notation. If $i = 1$, then $B_1 \supseteq T_1 (v_0 - v_0^2)$, so $\gr_\mc{F}(T_1 (v_0 - v_0^2)) = \barT_1 \bar{v}^2_0 \subseteq \barB_1$. For $i \geq 2$, by \Cref{lem: big image}, ${B}_i \supseteq {T}_i v_{i-1}^2$, thus $\barB_i \supseteq \barT_i \bar{v}_{i-1}^2$.
\end{proof}

Localization is an immensely useful construction in (non)commutative algebra to understand ideals of a ring. We thus define the following multiplicative set of $\barB_n$, over which we will localize. 

\begin{definition}\label{def: mul set barB}
    For $i \geq 1$, let 
    \[\barX_i = \{1, \bar{v}^{k}_{i-1} \mid k \geq 2\}\]
    be a multiplicative subset of $\barT_i$ and $\barT_i[\bart^{-1}]$. By \Cref{lem:barBn contains}, it is a multiplicative subset of $\barB_i$ and $\barC_i$ as well. 

    If $K$ is a $\barB_i$-module, we denote the corresponding localization of $K$ by
\begin{equation}
    K[\bar{v}^{-1}_{i-1}] := K \barX_i^{-1} = K \otimes_{\barB_i} (\barB_i\barX_i^{-1}).
\end{equation}
\end{definition}
By \Cref{lem:barBn contains}, the localizations of $\barB_n$ and $\barT_n$ with respect to $\barX$ coincide
\begin{align}
    \barB_n \barX_n^{-1} = \barT_n \barX_n^{-1} = \barT_n[\bar{v}_{n-1}^{-1}].
\end{align}

The last ingredient for our proof of \Cref{theo: Bn Cn ACC} is the following observation that the \emph{contraction of extension} of graded Poisson $\barB_i$-submodules of $\barT_i v^{\mathbf{a}}$ ``stabilizes'' for all ${\mathbf{a}} \in \NN^{n-i}$, which we will make precise by the following definition and \Cref{prop: stablize}. 
\begin{definition}
    Let $i \leq  n-1$ be integers and let $\mathbf{a} = (a_i, \dots, a_{n-1}) \in \NN^{n-i}$. 
    If $K$ is a graded Poisson $\barB_i$-submodule of $\barT_i \barv^\mathbf{a} $, let \emph{the contraction of extension of $K$} be
    \begin{equation}\label{eq: extend contract}
        \widehat{K} = K\barX_i^{-1} \cap \barT_i \barv^\mathbf{a}.
    \end{equation}
\end{definition}

There is a canonical Poisson structure on $K[\bar{v}^{-1}_{i-1}]=K\barX_i^{-1}$ making it a Poisson $\barT_i[\bar{v}^{-1}_{i-1}]$-module, see, e.g., the proof of \cite[Lemma 1.3]{Kaledin09Poisson} for the formula for $\barT_i[\bar{v}^{-1}_{i-1}]$. However, we do not need the Poisson structure of $\barT_i[\bar{v}^{-1}_{i-1}]$ for this Noetherian result.

\begin{lemma}\label{lem: localized Noetherian}
    Let $1 \leq i \leq n-1$ and $\mathbf{a} = (a_i, \dots, a_{n-1}) \in \NN^{n-i}$. If $K$ is a $\barB_i$-submodule of $\barT_i \barv^\mathbf{a}$, then $K[\bar{v}^{-1}_{i-1}]$ is a $\barT_i[\bar{v}_{i-1}^{-1}]$-submodule of $\barT_i[\bar{v}^{-1}_{i-1}]\barv^\mathbf{a}$. Furthermore, 
    \begin{equation*}
        \barT_i[\bar{v}^{-1}_{i-1}]\barv^\mathbf{a} \cong \barT_i[\bar{v}^{-1}_{i-1}]
    \end{equation*}
    are Noetherian $\barT_i[\bar{v}^{-1}_{i-1}]$-modules. The same statements hold for $\barC_i$ and $\barT_i[\bart^{-1}]$.
\end{lemma}
\begin{proof}
    Again, we only need to show the proof for $\barB_i$ and $\barT_i$. The fact that $K[\bar{v}^{-1}_{i-1}]$ is a $\barT_i[\bar{v}_{i-1}^{-1}]$-submodule of $\barT_i[\bar{v}^{-1}_{i-1}]\barv^\mathbf{a}$ is obvious. As $\barT_i[\bar{v}^{-1}_{i-1}] = \kk[\bar{t}, y, \bar{v}_0, \dots,\bar{v}_{i-1}][\bar{v}^{-1}_{i-1}]$ is a polynomial ring in finitely many variables, it is a Noetherian ring. Hence, the isomorphism of $\barT_i[\bar{v}^{-1}_{i-1}]$-modules
    \begin{equation}\label{eq: isomorphism}
        \barT_i[\bar{v}^{-1}_{i-1}]\barv^\mathbf{a} \to \barT_i[\bar{v}^{-1}_{i-1}] \quad r\barv^\mathbf{a} \mapsto r,
    \end{equation}
    makes $\barT_i[\bar{v}^{-1}_{i-1}]\barv^\mathbf{a}$ Noetherian. 
\end{proof}

\begin{remark}
    The isomorphism \eqref{eq: isomorphism} or the isomorphism $\barT_i\barv^\mathbf{a} \to \barT_i$ is \emph{not} a Poisson $\barT_i$ or $\barB_i$-modules homomorphism, as it does not respect the Poisson structure. For example 
    \begin{align*}
         \theta(\{v_0, \barv^\mathbf{a}\}) = \mc{A} \neq 0 =  \{v_0, \theta(\barv^\mathbf{a})\}.  
    \end{align*}
\end{remark}

However, we \emph{do} have a Poisson isomorphism between $\barT_i\barv^\mathbf{a}$ with the following Poisson $\barT_i$-modules. 

\begin{lemma}\label{lem: Poisson iso}
    Let $i+1 \leq n-1$ and let $\mathbf{b} = (b_{i+1}, \dots, b_{n-1}) \in \NN^{n-i-1}$ and $m \in \NN$. Then
    \begin{align*}
        \frac{\barT_{i+1} \barv_i^m \barv^\mathbf{b}}{\barT_{i+1} \barv_i^{m+1} \barv^\mathbf{b}} \cong \barT_i \barv_i^m \barv^\mathbf{b}
    \end{align*}
    as Poisson $\barT_i$-modules and $\barT_{i+1}$-modules. 
\end{lemma}
\begin{proof}
    For $c \geq d$, recall the projection $\Pi_{c, d}: T_{c} \surj T_d$ defined by $v_{\geq c} \mapsto 0$. This induces a projection $\overline{\Pi}_{c, d}: \barT_{c} \surj \barT_d$ defined by $\barv_{\geq c} \mapsto 0$. Consider the map 
    \begin{equation}\label{eq: the map theta}
        \theta: \barT_{i+1} \barv_i^m \barv^\mathbf{b} \to \barT_i \barv_i^m \barv^\mathbf{b}, \quad r\barv_i^m \barv^\mathbf{b} \mapsto \overline{\Pi}_{i+1, i}(r)\barv_i^m\barv^\mathbf{b}. 
    \end{equation}
    The Poisson projection $\overline{\Pi}_{i+1, i}: \barT_{i+1} \to \barT_i$ naturally makes $\barT_i \barv_i^m \barv^\mathbf{b}$ a (restricted) Poisson $\barT_{i+1}$-module. Since $\ker \overline{\Pi}_{i+1, i} = \barT_{i+1} \barv_i$ and $\theta(\barT_{i+1} \barv_i \barv_i^m \barv^\mathbf{b}) = 0$, we have $\barT_{i+1} \barv_i^m \barv^\mathbf{b}$ is a well-defined Poisson $\barT_i$-module.
       
    We now need to check that $\theta$ is a Poisson $\barT_i$-modules homomorphism and $\barT_{i+1}$-modules homomorphism. By \eqref{eq: Poisson structure not v0}, it is enough to check the Poisson structure with respect to $v_0$. Let $\mc{B} = \sum_{k=i+1}^{n-1} k a_k$. For all $r \in \barT_i$, 
    \begin{align*}
        \{v_0, \theta(r\barv_i^m \barv^\mathbf{b})\} = \{v_0, r\barv_i^m \barv^\mathbf{b}\} = \big ((\mc{B} + mi) r + \{v_0, r\}_{\barT_i} \big ) \barv_i^m \barv^\mathbf{b} = \theta(\{v_0, r\barv_i^m \barv^\mathbf{b} \}).
    \end{align*}
    It is easy to see that $\ker \theta = \barT_{i+1} \barv_i^{m+1} \barv^\mathbf{b}$, thus the lemma follows.
\end{proof}

The following proposition is the key ingredient in our induction step, and can be viewed as a stability statement.

\begin{proposition}\label{prop: stablize}
    Let $1 \leq i+1 \leq n -1$. Suppose for all $\mathbf{a} = (a_{i}, \dots, a_{n-1}) \in \NN^{n-i}$ that $\barT_{i} \barv^{\mathbf{a}}$ is a graded Poisson Noetherian $\barB_{i}$-module. Let $K$ be a graded Poisson $\barB_{i+1}$-submodule of $\barT_{i+1} \barv^{\textbf{b}}$ for $\textbf{b} = (b_{i+1}, \dots, b_{n-1}) \in \NN^{n-i-1}$. Then there exists a large enough $M \in \NN$ such that 
    \begin{align*}
        K[\bar{v}_{i}^{-1}] \cap \barT_{i+1} \bar{v}_{i}^{M} \barv^{\textbf{b}}= K \cap \barT_{i+1} \bar{v}_{i}^{M} \barv^{\textbf{b}}.
    \end{align*}
    The same statement holds for $\barC_{n}$ and $\barT_n[\bar{t}^{-1}]$.
\end{proposition}
\begin{proof}
    The proof is the same for $\barB_{n-1}$ and $\barC_{n-1}$, so we only show for $\barB_{n-1}$. Note that 
    \begin{align*}
        \widehat{K} = K[\bar{v}_{i}^{-1}] \cap \barT_{i+1} \barv^{\textbf{b}}= \{x \barv^{\textbf{b}} \in \barT_{i+1}\barv^{\textbf{b}} \mid x\bar{v}_{i}^{\ell} \barv^{\textbf{b}} \in K \text{ for some } \ell \geq 0\}.
    \end{align*}
    For all $m$, denote 
    \begin{align*}
        L_m = \{a \barv^{\textbf{b}} \in \barT_{i+1} \barv^{\textbf{b}} \mid a\barv_i^m \barv^{\textbf{b}} \in K\}.
    \end{align*}
    We first check that $\widehat{K}$ and $L_m$ are Poisson $\barB_{i+1}$-submodules of $\barT_{i+1}\barv^{\textbf{b}}$ for all $m$. It suffices to check for $L_m$ since $\widehat{K} = \bigcup_{m \geq 0} L_m$. Let $x\barv^{\textbf{b}} \in L_m$, so $x\bar{v}_{i}^{m} \barv^{\textbf{b}} \in K$. If $r \in \barB_{i+1}$ does not involve $\barv_0$, we have
    \begin{align*}
        \{r, x\barv^{\textbf{b}}\}_{\barT_{i+1}\barv^{\textbf{b}}}\bar{v}_{i}^{m} = \{r, x\bar{v}_{i}^{m} \barv^{\textbf{b}}\}_{\barT_{i+1}\barv^{\textbf{b}}} \in K,
    \end{align*}
    since $K$ is a Poisson $\barB_{i+1}$-submodule of $\barT_{i+1}\barv^{\textbf{b}}$. On the other hand, by \eqref{eq: new derivation axiom 2}, we have
    \begin{align*}
        \{v_0, x\bar{v}_{i}^{m} \barv^{\textbf{b}}\}_{\barT_{i+1}\barv^{\textbf{b}}} = \{v_0, x\barv^{\textbf{b}} \}_{\barT_{i+1}\barv^{\textbf{b}}} \bar{v}_{i}^{m}+  mi x\bar{v}_{i}^{m}\barv^{\textbf{b}} \in K.
    \end{align*}
    Thus for all $r \in \barT_{i+1}$, we have $\{r, x\barv^{\textbf{b}}\}_{\barT_{i+1}\barv^{\textbf{b}}}\bar{v}_{i}^{m} \in K$ and hence $\{r, x\barv^{\textbf{b}}\}_{\barT_{i+1}} \in L_m$. 
    
    By \Cref{lem:barBn contains}, $\barT_{i+1} \barv_i^2 \subseteq \barB_{i+1}$, we have $K \cap \barT_{i+1} \barv_i^{m}\barv^{\textbf{b}}\subseteq K \cap \barT \barv_i^{m+2}\barv^{\textbf{b}}$ and $L_{m} \subseteq L_{m+2}$ for all $m$. Recall the projection of Poisson $\barT_i$-modules
    \begin{align*}
    \theta : \barT_{i+1} \barv_i^{m} \barv^{\textbf{b}} \surj \barT_{i}\barv_i^{m} \barv^{\textbf{b}}
    \end{align*}
    in \eqref{eq: the map theta}. Thus by hypothesis, the projection of the chain $\{L_m\}$ stabilizes and there exists $M \in \NN$ so that $\theta(L_{M}) = \theta(L_{M+2}) = \dots$. Since $\barT_{i+1} \barv_{i}^3 \subseteq \barT_{i+1} \barv_{i}^2 \subseteq \barB_{i+1}$, we have $L_{m+j-2} \subseteq  L_{m+j+1} \subseteq L_{m+j+4}$, thus by renumbering we have without loss of generality that
    \[\theta(L_{M}) = \theta(L_{M+1}) = \theta(L_{M+2}) = \dots .\]
    Since $\widehat{K} = \bigcup_{i \in \NN} L_i$, we have 
    \begin{equation}\label{eq: pi widehat K}
        \theta(\widehat{K}) = \bigcup_{i \in \NN} \theta(L_i)= \theta(L_M).
    \end{equation}
    We claim that for $k \geq M$, 
    \begin{equation}
        \widehat{K} \cap \barT_{i+1} \barv_{i}^k  \barv^{\mathbf{b}} = K \cap \barT_{i+1} \barv_{i}^k \barv^{\mathbf{b}}.
    \end{equation}
    Suppose the contrary, there exists $z = w\barv_{i}^k  \barv^{\mathbf{b}} \in \widehat{K} \cap \barT_{i+1} \barv_{i}^k  \barv^{\mathbf{b}}  \setminus K \cap \barT_{i+1} \barv_{i}^k  \barv^{\mathbf{b}}$, and we can assume that $w$ is homogeneous since $K, \widehat{K}$ are graded. WLOG, we can also assume $\deg w$ is minimal among elements with this property. This means $w \notin \barT_{i+1} \barv_{i}$, as otherwise we have $z = w' \barv_{i}^{k+1} \barv^{\mathbf{b}}$ with $\deg w' < \deg w$. Since $w\barv^{\mathbf{b}} \in \widehat{K}$, by \eqref{eq: pi widehat K}, $\theta(w\barv^{\mathbf{b}}) \in \theta( \widehat{K}) = \theta(L_M) = \theta(L_k)$. So there exists a homogeneous element $c\barv^{\mathbf{b}} \in \barT_{i+1}\barv^{\mathbf{b}}$ with $\theta(c) = \theta(w)$ and $c \barv_i^k \barv^{\mathbf{b}} \in K$. Thus there exists some $d \in \barT_{i+1}$ such that 
    \begin{align*}
        w-c = d \barv_i, \quad d \barv_i^{k+1} \barv^{\mathbf{b}} = a\barv_i^k \barv^{\mathbf{b}} - c\barv_i^k \barv^{\mathbf{b}} \in \widehat{K} \cap \barT \barv_i^k \barv^{\mathbf{b}}.
    \end{align*}
    Since $\deg d = \deg w - 1$, we have $d\barv_i^{k+1} \barv^{\mathbf{b}} \in K \cap \barT v^k$ by minimality of $w$. Thus $w\barv_i^k \barv^{\mathbf{b}} = d \barv_i^{k+1} \barv^{\mathbf{b}} + c\barv_i^k \barv^{\mathbf{b}} \in K \cap \barT \barv_i^k \barv^{\mathbf{b}}$, a contradiction.
\end{proof}

\begin{remark}\label{remark: the choice of filtration}
    The proof of \Cref{prop: stablize} is inspired by the argument in \cite[Proposition 8.7]{Rogalski2015Sklyanin}. The key ingredient in this argument is a $\ZZ_{\geq 1}$-grading. Thus, we cannot work directly with $B_n$ and $T_n$ or use the canonical grading induced from $\W{-1}$ with $\deg e_i = i$. 
\end{remark}

With that, we can show the induction step for \Cref{prop: barBn ACC graded stronger} as follows. 

\begin{proof}[Proof of \Cref{prop: barBn ACC graded stronger}]
    We will show the theorem by induction on $i$. For the base case $i = 0$, we need to show that $\barT_0 v^{\mathbf{a}}$ is a graded Poisson Noetherian $\barB_0$-module for all $\mathbf{a} \in \NN^n$. The $\barB_0$-module isomorphism $\barT_0 v^{\mathbf{a}} \cong \barT_0$ is in fact also a Poisson $\barB_0$-module isomorphism. Thus, by \Cref{Poisson B0}, the base case is shown.

    Assume the induction hypothesis (IH1) that $\barT_i v^{\mathbf{a}}$ is a graded Poisson Noetherian $\barB_i$-module for all $\mathbf{a} = (a_i, \dots, a_{n-1})\in \NN^{n-i}$. We need to show that for all $\mathbf{b} = (b_{i+1}, \dots, b_{n-1}) \in \NN^{n-i-1}$ that $\barT_{i+1} \barv^{\mathbf{b}}$ is a graded Poisson Noetherian $\barB_{i+1}$-module.

    Let $K_1 \subseteq K_2 \subseteq \dots$ be graded Poisson $\barB_{i+1}$-submodules of $\barT_{i+1} \barv^{\mathbf{b}}$. By \Cref{lem: localized Noetherian}, $\barT_{i+1}[v_{i}^{-1}] \barv^{\mathbf{b}}$ is a Noetherian $\barT_{i+1}[v_{i}^{-1}]$-module and each $K_i[\bar{v}_{n-1}^{-1}]$ is a $\barT_{i+1}[v_{i}^{-1}]$-submodule of $\barT_{i+1}[v_{i}^{-1}] \barv^{\mathbf{b}}$. Thus, there exists $N$ such that 
    \[\widehat{K_N} = \widehat{K_{N+1}} = \dots\]
    By \Cref{prop: stablize}, there exists $m$ large enough such that 
    \[\widehat{K_N} \cap \barT_{i+1} \bar{v}_{i}^m \barv^{\mathbf{b}} = {K_N} \cap \barT_{i+1} \bar{v}_{i}^m \barv^{\mathbf{b}}. \]
    
    We will now show by induction on $\ell$ that 
    $\frac{\barT_{i+1} \barv^{\mathbf{b}}}{\barT_{i+1}\bar{v}_{i}^\ell \barv^{\mathbf{b}}}$ is a graded Poisson Noetherian $\barB_{i+1}$-module for all $\ell \in \NN$. The base case $\ell=1$, we have $\frac{\barT_{i+1} \barv^{\mathbf{b}}}{\barT_{i+1}\bar{v}_{i} \barv^{\mathbf{b}}} \cong \barT_{i} \barv^{\mathbf{b}}$ as Poisson $\barB_i$-module by \Cref{lem: Poisson iso}. Thus, by (IH1), it is a graded Poisson Noetherian $\barB_{i}$-module. It is also a graded  Noetherian Poisson $\barB_{i+1}$-module since any graded Poisson $\barB_{i+1}$-submodule of it is a Poisson $\barB_{i}$-module.

    Assuming the induction hypothesis (IH2) that $\frac{\barT_{i+1} \barv^{\mathbf{b}}}{\barT_{i+1}\bar{v}_{i}^\ell \barv^{\mathbf{b}}}$ is a graded Poisson Noetherian $\barB_{i+1}$-module. We have a short exact sequence of Poisson $\barB_{i+1}$-modules
    \begin{equation}\label{eq: SES barBi}
        0 \to \frac{\barT_{i+1} \bar{v}_{i}^\ell  \barv^{\mathbf{b}}}{\barT_{i+1}\bar{v}_{i}^{\ell+1} \barv^{\mathbf{b}}} \to \frac{\barT_{i+1} \barv^{\mathbf{b}}}{\barT_{i+1}\bar{v}_{i}^{\ell+1} \barv^{\mathbf{b}}} \to \frac{\barT_{i+1} \barv^{\mathbf{b}}}{\barT_{i+1}\bar{v}_{i}^{\ell} \barv^{\mathbf{b}}} \to 0.
    \end{equation}
    By \Cref{lem: Poisson iso}, $\frac{\barT_{i+1} \bar{v}_{i}^\ell  \barv^{\mathbf{b}}}{\barT_{i+1}\bar{v}_{i}^{\ell+1} \barv^{\mathbf{b}}} \cong \barT_{i} \bar{v}_{i}^\ell  \barv^{\mathbf{b}}$ as a graded Poisson $\barB_i$-module. Thus it is Noetherian by (IH1), and by the same argument as before, is a graded Poisson Noetherian $\barB_{i+1}$-module. On the other hand, $\frac{\barT_{i+1} \barv^{\mathbf{b}}}{\barT_{i+1}\bar{v}_{i}^{m} \barv^{\mathbf{b}}}$ is a graded Noetherian $\barB_{i+1}$-module by (IH2). Thus the claim follows by induction. In particular, $\frac{\barT_{i+1} \barv^{\mathbf{b}}}{\barT_{i+1}\bar{v}_{i}^m \barv^{\mathbf{b}}}$ is a graded Poisson Noetherian $\barB_{i+1}$-module.

    Let $\sigma : \barT_{i+1} \barv^{\mathbf{b}} \to \barT_{i+1} \barv^{\mathbf{b}}/{\barT_{i+1}\bar{v}_{i}^m \barv^{\mathbf{b}}}$ be the canonical projection of Poisson $\barB_{i+1}$-modules. We have $\sigma(K_1) \subseteq \sigma(K_2) \subseteq \dots$. Thus, since $\barT_{i+1} \barv^{\mathbf{b}}/{\barT_{i+1}\bar{v}_{i}^m \barv^{\mathbf{b}}}$ is a graded Poisson Noetherian $\barB_{i+1}$-module,
    \[\sigma(K_M) = \sigma(K_{M+1}) = \dots \]
    for some $M$. By the short exact sequence 
    \[0 \rightarrow {K_\ell} \cap \barT_{i+1}\barv_i^m \barv^{\mathbf{b}} \rightarrow K_\ell  \rightarrow \sigma(K_\ell) \rightarrow 0,\]
    for $\ell \geq \max(M, N)$, we have $K_\ell = K_{\max(M, N)}$. Thus $\barT_{i+1}\barv^{\mathbf{b}}$ is a graded Poisson Noetherian $\barB_{i+1}$-module.
\end{proof}

\begin{remark}\label{remark: why induction like this}
    The isomorphism $\frac{\barT_{i+1} \bar{v}_{i}^m  \barv^{\mathbf{b}}}{\barT_{i+1}\bar{v}_{i}^{m+1} \barv^{\mathbf{b}}} \cong \barT_{i} \bar{v}_{i}^m  \barv^{\mathbf{b}}$ in the short exact sequence \eqref{eq: SES barBi} is the reason we need to prove need a stronger induction hypothesis and stronger statement in \Cref{prop: barBn ACC graded stronger}.
\end{remark}

Then \Cref{theo: Bn Cn ACC} follows naturally as follows. 
\begin{proof}[Proof of \Cref{theo: Bn Cn ACC}]
    From \Cref{prop: barBn ACC graded stronger}, $\barT_n$ is a graded Poisson Noetherian $\barB_n$-module, so $\barB_n \subseteq \barT_n$ satisfies ACC on graded Poisson ideals. Then by \Cref{prop: ACC graded induce ACC }, $T_n$ is a Noetherian $(B_n, B_n)$-bimodule and $B_n$ satisfies ACC on two-sided ideals. 
\end{proof}

\Cref{theo: Bn Cn ACC} makes great progress and gives significant evidence to \Cref{conj: UW ACC}. However, we still do not know whether $\Ua(W), \Ua(\W{-1})$ and $\Ua(\Vir)$ satisfy the ACC on two-sided ideals, and this is the topic of future research. \Cref{theo: Bn Cn ACC} suggests that a feasible to tackle that problem is the following conjecture.

\begin{conjecture}
    The Poisson algebras $\Sa(\W{-1})$, $\Sa(W)$ and $\Sa(\Vir)$ satisfy ACC on graded Poisson ideals.
\end{conjecture}

\section{The prime spectrum of \texorpdfstring{$B_n$}{Bn}}\label{sec: prime Bn}
In this section, we completely characterize the prime spectrum of $B_n$. Our analysis is a powerful combination of noncommutative localization and induction on homomorphic image by the surjections $\Pi_{n, m}|_{B_n}: B_n \surj B_{m}$ for $n \geq m$. Interestingly, the fact that $B_n$ satisfies ACC on two-sided ideals in \Cref{theo: Bn Cn ACC} shows the contraction of extension of a prime ideal of $B_n$ ``stabilizes'', which we prove in Subsection \ref{subsec: localization on prime}. Then in Subsection \ref{subsec: Dx Spec and Prim gn}, we use the Dixmier map $\Dx_{\g_n}$ for the finite-dimensional solvable Lie algebra $\g_n$ to describe the prime and primitive ideals of a localization of $T_n$. This demonstrates the application of the Dixmier map $\Dx_{\g_n}$ in noncommutative ring. Finally in Subsection \ref{subsec: restriction prime},we completely describe the prime ideals of $B_n$, all but one coming from restricting prime ideals of $T_n$. In fact, we show that restriction induces an ``almost homeomorphism'' between $\Spec T_n$ and $\Spec B_n$.

\subsection{Localization and operators on prime ideals of \texorpdfstring{$B_n$}{Bn}}\label{subsec: localization on prime}
In this subsection, we localize $B_n$ and $T_n$ at the set $X_n$ defined in \Cref{def: Xn} below. This sclassifies a large class of prime ideals of $B_n$, while the other class of prime ideals can be understood via the projections $\Pi_{n, n-1}|_{B_n}: B_n \surj B_{n-1}$. The main result in this subsection is \Cref{prop: extension contraction n} that the contraction of extension of a prime ideal $P$ of $B_n$ stabilizes, i.e., if $P \cap X_n = \emptyset$, then 
\begin{equation}
    PX_n^{-1} \cap B_n = P.
\end{equation}

Recall the multiplicative subset $\barX_i =\{1, \bar{v}_{i-1}^m \mid m \geq 2\}$ of $\barB_i$ in \Cref{def: mul set barB}. Localization over $\barX_i$ plays an important role in the proof of \Cref{theo: Bn Cn ACC}. We now consider the following multiplicative set of the filtered ring $T_n$, over which we will localize. 
\begin{definition}\label{def: Xn}
    Let 
    \[X_1 = \{1, v_0^a(v_0-1)^b \mid a, b \geq 1\}\]
    be a multiplicative set of $T_1$. For $n \geq 2$, let 
    \begin{equation}
        X_n = \{v_{n-1}^{2k} \mid k \in \NN\}
    \end{equation}
    be a multiplicative set of $T_n$.
\end{definition}

By \Cref{lem: big image}, $X_n$ is also a multiplicative set of regular elements of $B_n$ for all $n$. Although $v_{n-1}^k$ is not normal in $B_n$, we show that $X_n$ is still a left and right Ore set of $B_n$, which allows us to localize $B_n$ and $T_n$ at $X_n$. We furthermore prove that their localizations coincide, and are Noetherian rings. 
\begin{lemma}\label{localization of Bn}
    For all $n$, $v_{n-1}$ is normal in $T_n$ and $X_n$ is still a left and right Ore set of $B_n$. Moreover, 
    \begin{align}
    B_1X_1^{-1} &= T_1X_1^{-1} = A_1 \otimes \kk[v_0, v_0^{-1}, (v_0-1)^{-1}], \label{eq: localization of B1}\\
    B_nX_n^{-1} &= T_nX_n^{-1} = A_1 \otimes \Ua(\g_n)[v_{n-1}^{-1}], \label{eq: localization of Bn}
    \end{align}
    are (left and right) Noetherian rings. 
\end{lemma}
\begin{proof}
    Since $v_{n-1}$ commutes with $t, \del, v_1, \dots, v_{n-1}$, it suffices to check it on $v_0$. We have
    \begin{align*}
        v_0v_{n-1} = v_{n-1} (v_0 + n-1),
    \end{align*}
    so $v_{n-1}$ is normal in $T_{n}$. We now show that $X_n$ is an Ore set of $B_n$. If $n=1$, this follows from $v_0 \in Z(T_1) \subseteq Z(B_1)$, so we assume $n \geq 2$. For all $r \in B_n$ and $k \geq 2$, as $v_{n-1}^{k-2}$ is normal in $T_n$, we have
    \begin{equation}\label{eq: Ore help}
        r v_{n-1}^k  =  r v_{n-1}^{k-2}v_{n-1}^{2}  = v_{n-1}^{k-2} \widetilde{r} v_{n-1}^2 \quad \text{ for some } \widetilde{r} \in T_n.
    \end{equation}
    By \Cref{lem: big image}, $T_n v_{n-1}^2 \subseteq B_n$, so $X_n$ is a left Ore set. Similarly, for $r \in B_n$, we have
    \begin{equation}\label{eq: Ore help 2}
        v_{n-1}^k r = v_{n-1}^2 r^{\sharp}v_{n-1}^{k-2} \text{ for some }r^{\sharp} \in T_n.
    \end{equation} 
    As $v_{n-1}^2 T_n = T_n v_{n-1}^2 \subseteq B_n$, \eqref{eq: Ore help 2} shows that $X_n$ is a right Ore set. Thus the left and right ring of fractions $B_nX_n^{-1}$ and $X_n^{-1}B_n$ exist and coincide. Since $T_n v_{n-1}^2 \subseteq B_n$, we have $B_nX_n^{-1} = T_nX_n^{-1}$. Thus \eqref{eq: localization of B1} and \eqref{eq: localization of Bn} follow. Since $T_n$ is Noetherian, by \cite[Corollary 10.16 a]{goodearl1989introduction}, $T_nX_n^{-1} = B_n X_n^{-1}$ is also Noetherian.
\end{proof}

The following lemma explains why we localize at $X_n$. 

\begin{lemma}\label{lem: prime two cases}
    For all $n \geq 1$ and $P \in \Spec B_n$, then either $P \supseteq \ker \Pi_{n, n-1}|_{B_n}$ or $P \cap X_n = \emptyset$.
\end{lemma}
\begin{proof}
    Let $P$ be a prime ideal such that $P \cap X_n \neq \emptyset$. We consider the following cases.
    
    If $n \geq 2$, then there exists $k \geq 1$ such that $v_{n-1}^{2k} \in P$. Then by \cite[Proposition 3.12]{PhamTimmins2025}, $T_n v_{n-1}^{2k} \subseteq P$. If $k \geq 2$, since
    \begin{align*}
    T_n v_{n-1}^{2k} = T_n v_{n-1}^2 T_{n-1} v_{n-1}^{2k-2} \subseteq P,
    \end{align*}
    we have $T_n v_{n-1}^2 \subseteq P$ or $T_{n-1} v_{n-1}^{2(k-1)} \subseteq P$. Iteratively, we have $T_n v_{n-1}^2 \subseteq P$ and by \cite[Corollary 4.21]{PhamTimmins2025}, $\ker \Pi_{n, n-1} = T_n v_{n-1}^2$.

    If $n=1$, then since $T_1v_0(v_0-1) \subseteq B_1$, there exists $k \geq 1$ such that $T_1 v_0^k (v_0-1)^{k} \subseteq P$. A similar argument as above shows that $T_1 v_0 (v_0-1) = \ker \Pi_{1, 0} \subseteq P$.
\end{proof}

For any ring homomorphism there are useful operations known as contraction and extension relating ideals of the domain and codomain. We record the central notions below in the setting of localization of noncommutative rings for completeness. For more details see, e.g., \cite[Chapter 10]{goodearl1989introduction}.

\begin{definition}
    Given an ideal of $K$ of $B_n$, then since $B_n X_n^{-1}$ is Noetherian by \Cref{localization of Bn}, $K X_n^{-1}$ is an ideal of $B_n X_n^{-1}$ by \cite[Theorem 10.18]{goodearl1989introduction}. We call $K X_n^{-1}$ the \emph{extension} of $K$. 
    
    On the other hand, if $H$ is an ideal of $B_n X_n^{-1}$, then $H \cap B_n$ is an ideal of $B_n$, which is known as the \emph{contraction} of $H$. 

    For an ideal $K$ of $B_n$, we thus let 
    \begin{equation}
        \widehat{K} = KX^{-1}_n \cap B_n,
    \end{equation}
    be the \emph{contraction of extension} of $K$ with respect to $X_n$. 
\end{definition}

We now show that the contraction of extension of a prime ideal $P$ of $B_n$ that are disjoint from $X_n$ equals $P$ itself. The case $n=1$ is easier and different from $n \geq 2$, so we show it first to explain the idea. Recall that for a commutative ring $R$ and a multiplicative set $X$, if $I$ is a prime ideal of $R$ such that $I \cap X = \emptyset$, then $IX^{-1}$ is a prime ideal of $RX^{-1}$ and $IX^{-1} \cap R = I$. The case $n=1$ mimics exactly the commutative case above since $v_0$ is central in $T_1$ and $B_1$.

\begin{proposition}\label{prop: extension contraction 1}
    For all prime ideals $P$ of $B_1$ disjoint from $X_1$, we have 
    \begin{align*}
        PX_1^{-1} \cap B_1 = P.
    \end{align*} 
\end{proposition}
\begin{proof}
    To simplify notation, in this proof, let $v$ denote $v_0$. We first show that $ \widehat{P} (v^2-v)^m \subseteq P$ for some $m$. Note that the contraction of extension of $P$ is  
    \begin{equation}
      \widehat{P} = \{p \in B_1 \mid \text{there exists } a,b \text{ such that } p v^a(v-1)^b \in P\}. 
    \end{equation}
    As $B_1X_1^{-1} = T_1X_1^{-1}$ is Noetherian, by \cite[Theorem 10.18]{goodearl1989introduction}, $PX_1^{-1}$ is an ideal of $T_1X_1^{-1}$. Thus, $\widehat{P} = PX_1^{-1} \cap B_1$ is an ideal of $B_1$. For all $i \geq 2$, define  
    \begin{equation}\label{eq: power of v2-v}
        P_i = \{p \in B_1 \mid p (v^2-v)^i \in P\}. 
    \end{equation}
    It is easy to see that $P_i$ are two-sided ideals of $B_1$ as $v$ is central in $B_1$. We have the chain of inclusions $P_1 \subseteq P_2 \subseteq \dots$, which stabilizes as $B_1$ satisfies ACC on two-sided ideals by \Cref{theo: ACC}. Thus, there exists $m$ such that $ \widehat{P}(v^2-v)^m \subseteq P$, so $\widehat{P} \subseteq P$ by assumption. Hence, $P = \widehat{P}$.
\end{proof}

The case $n \geq 2$ is trickier since the ideal $P_i$ in \eqref{eq: power of v2-v} is now only a left ideal of $B_n$ and might not be a right ideal in general. However, we can still obtain the same result.
\begin{proposition}\label{prop: extension contraction n}
    Let $n \in \NN$ and $P$ be a prime ideal of $B_n$ such that $P \cap X_n = \emptyset$. Then
    \begin{equation*}
        PX_n^{-1} \cap B_n = P.
    \end{equation*} 
\end{proposition}

\begin{proof}
    The case $n =1$ is \Cref{prop: extension contraction 1}, thus we assume $n \geq 2$ from now on. To simplify notation, in this proof, let $v$ denote $v_{n-1}$. We first show that $ \widehat{P} v^m \subseteq P$ for some $m$. Note that the contraction of extension of $P$ is  
    \begin{equation}
      \widehat{P} = \{p \in B_n \mid \text{there exists } k \text{ such that } p v^k \in P\}. 
    \end{equation}
    As $B_nX_n^{-1} = T_nX_n^{-1}$ is Noetherian, by \cite[Theorem 10.18]{goodearl1989introduction}, $PX_n^{-1}$ is an ideal of $T_nX_n^{-1}$. Thus, $\widehat{P} = PX_n^{-1} \cap B_n$ is an ideal of $B_n$.
    For all $i \geq 2$, let  
    \begin{equation}
        P_i = \{p \in B_n \mid p v^i \in P\} 
        \quad \text{and} \quad Q_i = P_iB_n.
    \end{equation}
    It is easy to see that $P_i$ are left ideals of $B_n$ for all $i \in \NN$. Thus, $Q_i \subseteq B_n$ are two-sided ideals of $B_n$ for all $i$. Moreover, 
    \begin{equation}
        P_k \subseteq P_{k+2} \subseteq \dots \Rightarrow Q_k \subseteq Q_{k+2} \subseteq \dots. 
    \end{equation}
    By \Cref{theo: ACC}, $B_n$ satisfies ACC on two-sided ideals, so there exists $m$ such that 
    \[Q_m = Q_{m+2} = \dots \quad \text{ and by renumbering } Q_m = Q_{m+1} = \dots ,\]
    since $P_{k} \subseteq P_{k+3}$ and $Q_{k} \subseteq Q_{k+3}$ for all $k$. Then as  
    \begin{align}
        \widehat{P} = \widehat{P}B_n = \big (\bigcup P_k \big ) B_n = \bigcup Q_k,
    \end{align}
    we have $\widehat{P} = Q_m$. Equation \eqref{eq: Ore help} implies that $v^m B_n \supseteq B_n v^{m+2}$, so 
    \begin{equation}
        P = PB_n \supseteq P_m v^m B_n \supseteq P_m B_n v^{m+2} = Q_m v^{m+2} = \widehat{P} v^{m+2} \supseteq \widehat{P} T_n v^{m+4},
    \end{equation}
    where the last containment follows from \Cref{lem: big image}. Since $P$ is prime and $T_n v^{m+4} \nsubseteq P$ by assumption, $\widehat{P} \subseteq P$. So $\widehat{P} = P$.
\end{proof}

The following corollary follows immediately, which could be seen as the stronger version of \cite[Proposition 10.20]{goodearl1989introduction} adapted to our non-Noetherian ring $B_n$. We furthermore characterize the topology of the maps.
\begin{corollary}
    Contraction and extension give inverse homeomorphisms between prime ideals of $B_1$ not containing $v_0 (v_0-1)$ and prime ideals of $A_1 \otimes \kk[v_0, v_0^{-1}, (v_0 -1)^{-1}]$. In other words, 
    \begin{align}
        \Spec T_1[v_0^{-1}, (v_0 -1)^{-1}] \cong \Spec B_n \setminus \mathcal{V}(T_1 v_0(v_0 -1)),
    \end{align}
    which is an open set of $B_1$.

    For $n \geq 2$, contraction and extension give inverse homeomorphisms between prime ideals of $B_n$ not containing $v_{n-1}^2$ and prime ideals of $A_1 \otimes \Ua(\g_n)[v_{n-1}]$. In other words, 
    \begin{equation}\label{eq: homeomorphism Bn Tn localized}
        \Spec T_n[v_n^{-1}] \cong \Spec B_n \setminus \mathcal{V}(T_n v_{n-1}^2),
    \end{equation}
    which is an open set of $B_n$.
\end{corollary}
\begin{proof}
    The proof in the case of $n=1$ is similar to $n \geq 2$, so we only provide for $n \geq 2$. Then \Cref{prop: extension contraction n} and \cite[Proposition 10.17 a]{goodearl1989introduction} imply that contraction and extension give inverse bijections between $\Spec T_n[v_n^{-1}]$ and $\Spec B_n \setminus \mathcal{V}(T_n v_{n-1}^2)$. For an ideal $I$ and a prime ideal $P$ of $B_n$ both disjoint from $X_n$, we have 
    \begin{align}
        I \subseteq P \subseteq B_{n-1} \Rightarrow IX^{-1} \subseteq PX^{-1} \subseteq T_n[v_n^{-1}],
    \end{align} 
    which implies that the extension map $\Spec B_n \setminus \mathcal{V}(T_n v_{n-1}^2) \to \Spec T_n[v_n^{-1}]$ is continuous. 
    
    On the other hand, by \cite[Proposition 10.17]{goodearl1989introduction} all ideals of $T_n[v_n^{-1}]$ are extensions of $B_n$. Let $IX^{-1}$ and $PX^{-1}$ be an ideal and a prime ideal of $T_n[v_n^{-1}]$ respectively such that $IX^{-1} \subseteq PX^{-1}$. Then by \Cref{prop: extension contraction n},
    \begin{align*}
        P = PX^{-1} \cap B_n \supseteq IX^{-1} \cap B_n \supseteq I,
    \end{align*} 
    which shows that contraction map is continuous as well. Thus, we obtain the homeomorphisms in \eqref{eq: homeomorphism Bn Tn localized}.
\end{proof}

\subsection{The spectra \texorpdfstring{$\Spec T_nX_n^{-1}$}{Spec localized gn} and \texorpdfstring{$\Prim T_nX_n^{-1}$}{Prim localized gn} from the Dixmier map for \texorpdfstring{$\g_n$}{gn}}\label{subsec: Dx Spec and Prim gn}
In this subsection, we study the prime and primitive spectrum of $T_nX_n^{-1}$. The main tool we will use is the Dixmier map $\Dx_{\g_n}$ for $\g_n$ and the results on coadjoint orbits of $\g_n^*$ characterized in \cite{pham2025orbit}. Recall the following useful result that reduces the study of prime and primitive ideals of $T_n[v_{n-1}^{-1}]$ to $\Ua(\g_n)[v_{n-1}^{-1}]$.

\begin{lemma}\label{lem:idealcor}
    For a ring $R$, let $\mathcal{I}(R)$ denote the ideal lattice of $R$. The map 
    \begin{equation}\label{eq:correspondence}
    \iota_*: \mathcal{I}(\Ua(\g_n)) \rightarrow \mathcal{I}(T_n), \quad 
    J \longmapsto {A}_1 \otimes J, 
    \end{equation}
    is a well-defined bijection between ideals of $\Ua(\g_n)$ and ideals of $T_n$. Moreover, it restricts to homeomorphisms
    \begin{equation*}
    \Spec \Ua(\g_n) \rightarrow \Spec T_n, \quad \Prim \Ua(\g_n) \rightarrow \Prim T_n,
    \end{equation*}
    which we denote also by $\iota_*$.
\end{lemma}
\begin{proof}
    This follows from \cite[Lemma 9.6.9]{mcconnell2001noncommutative} since $A_1$ is a simple ring with trivial center $\kk$ and \cite[Proposition 3.6.11]{phamthesis}.
\end{proof}

The rings $\Ua(\g_n)[v_{n-1}^{-1}]$ and $T_n[v_{n-1}^{-1}]$ are well-behaved Noetherian noncommutative rings, since they satisfy the Nullstellensatz over $\kk$ (\cite[Corollary 9.1.8]{mcconnell2001noncommutative}) and are Jacobson rings.
\begin{lemma}\label{lem: Jacobson rings 1}
    The rings $\Ua(\g_n)[v_{n-1}^{-1}]$ and $T_n[v_{n-1}^{-1}]$ 
    satisfy the radical property and furthermore are Jacobson rings. 
\end{lemma}
\begin{proof}
    Every prime ideal of $\Ua(\g_n)$ is completely prime, so prime ideals $P$ of $\Ua(\g_n)[v_{n-1}^{-1}]$ is also completely prime. Thus prime ideals of $T_n[v_{n-1}^{-1}]$ are completely prime as well. To simplify notation, $R$ denote $\Ua(\g_n)[v_{n-1}^{-1}]$ and $T_n[v_{n-1}^{-1}]$ in this proof. 
    
    Let $I$ be an ideal of $R$. Since $R/I$ is a finitely generated $\kk$-algebra, $\dim_\kk R/I$ is countable and thus by \cite[Corollary 4]{Amitsur56}, the Jacobson radical $J(R/I)$ is nil, and thus $R$ satisfies the radical property. 

    Since $R$ is Noetherian, all prime factor rings are right Goldie, $R$ is a Jacobson ring by \cite[Lemma 1.2 (ii)]{mcconnell2001noncommutative}.    
\end{proof}

Thus, it suffices to study the prime and primitive spectrum of $\Ua(\g_n)[v_{n-1}^{-1}]$. The case $n =1$ is easier and requires slightly different notation, which we will explain first below. 

\begin{proposition}\label{Spec B1}
    The prime spectrum of $B_1$ is
    \begin{equation*}
        \Spec B_1 = \{0,\  \Psi_1(\Ua(\W{-1})_+),\  T_1(v_0-v_0^2),\  T_1 (v_0 - \alpha) \cap B_1 \mid \alpha \neq 0, 1\},
    \end{equation*}
    which are all completely prime ideals. 
\end{proposition}

\begin{proof}
    Let $P \neq 0$ be a prime ideal of $B_1$. By \Cref{lem: prime two cases}, $P \supseteq T_1 v_0 (v_0-1) = \ker \Pi_{1, 0}\cap B_1$ or $P \cap X_1 = \emptyset$. In the first case, by \Cref{prop:B0ideals},
\[P = \begin{cases}
        \ker \Pi_{1, 0} = T_1 v_0 (v_0-1) \\ 
        \Pi_{1, 0}^{-1}(A_1 \del) = \Psi_1(\Ua(\W{-1})_+).
    \end{cases}\]

    Now if $P \cap X_1 = \emptyset$, then by \cite[Theorem 10.18]{goodearl1989introduction}, $PX_1^{-1}$ is a proper prime ideal of $B_1X_1^{-1} = T_1X_1^{-1} = A_1 \otimes \kk[v_0, v_0^{-1}, (v_0-1)^{-1}]$ by \Cref{localization of Bn}. By \Cref{lem:idealcor}, 
    \begin{align*}
        PX_1^{-1} = T_1X_1^{-1} (v_0-\alpha) \quad \text{for some} \quad \alpha \in \kk, \alpha \neq 0, 1.
    \end{align*}
    By \Cref{prop: extension contraction 1}, 
    \begin{equation}
        P = PX_1^{-1} \cap B_1 =  T_1X_1^{-1} (v_0-\alpha) \cap B_1 = T_1 (v_0-\alpha) \cap B_1,
    \end{equation}
    which is easily seen to be a completely prime ideal of $B_1$. By \cite[Lemma 2.1]{conley2007}, $B_1/(T_1 (v_0 - \alpha) \cap B_1)$ and $B_1/T_1(v_0-v_0^2)$ are nonisomorphic, so $T_1 (v_0 - \alpha) \cap B_1 \neq T_1(v_0-v_0^2)$.
\end{proof}

\begin{remark}
    The Krull dimension of $B_0$ is 1 and the Krull dimension of $T_0$ is 0. Similarly, the Krull dimension of $B_1$ is 2, and the Krull dimension of $T_1$ is 1.
\end{remark}

From now on in this subsection, assume $n \geq 2$. By \Cref{lem:idealcor}, we reduced the question of studying prime and primitive ideals $T_nX_n^{-1}$ to those of $\Ua(\g_n)[v_{n-1}^{-1}]$. Recall that $\g_n$ is a finite-dimensional solvable Lie algebra with the algebraic adjoint group $G_n$ constructed in \cite[Definition 7.3]{pham2025orbit}. By \cite[Theorem 6.5.12]{dixmier1996enveloping}, we have a homeomorphism 
    \begin{equation}\label{eq: Dx gn}
        \Dx_{\g_n}: \g_n^*/G_n \cong \Prim \Ua(\g_n), \quad G_n \cdot \bchi \mapsto Q'_{\bchi}.
    \end{equation}
We now explain and recall the construction of $Q'_{\bchi}= \Ann_{\Ua(\g_n)} M_{\bchi}$ in \cite[Theorem 6.5.12]{dixmier1996enveloping} and \cite{pham2025orbit}. For any element $\bchi \in \g_n^*$, there exists a Vergne polarization $\p_{\bchi} = \W{\left\lfloor \frac{m}{2}\right\rfloor +1}/\W{n}$ for $\bchi$ (\cite[Proposition 6.5]{pham2025orbit}). A polarization for $\bchi$ is a maximal isotropic subspace of $\g$ that is also a subalgebra. Thus $\bchi$ is a character of $\p_{\bchi}$ and defines a one-dimensional representation $\kk_{\bchi}$ of $\p_{\bchi}$. We can induce it to a representation $M_{\bchi} = \Ua(\g_n) \otimes_{\Ua(\p_{\bchi})} \kk_{\bchi}$ of $\g_n$ and let $Q'_{\bchi}$ be its annihilator. Although the map $\Dx_{\g_n}$ and $Q'_{\bchi}$ are very interesting (we refer the reader to \cite{pham2025orbit} or \cite[Chapter 3]{phamthesis} for more details), for our purpose, the following lemma suffices. 

\begin{lemma}\label{vn annihilator}
    Let $\bchi \in \g^*_n$. Then 
    \begin{equation*}
        v_{n-1} \in Q'_{\bchi} \quad \text{if and only if} \quad \bchi(v_{n-1}) = 0.
    \end{equation*}
\end{lemma}
\begin{proof}
    If $v_{n-1} \in Q'_{\bchi}$, then in particular $\bchi(v_{n-1}) = v_{n-1} \cdot 1_{\bchi} = 0$. Now suppose $\bchi(v_{n-1}) = 0$, since $v_{n-1}$ is normal in $\Ua(\g_n)$, for any $r \in \Ua(\g_n)$, there exists $r' \in \Ua(\g_n)$ such that
    \begin{align*}
    v_{n-1} r 1_\chi = r' v_{n-1} 1_\chi = 0.
    \end{align*}
    Thus $v_{n-1} \in Q'_{\bchi}$. 
\end{proof}

Consequently, we completely describe the primitive spectrum $\Prim T_n[v_{n-1}^{-1}]$ as follows.

\begin{proposition}\label{theorem: primitive ideals gn}
    For all $ n \geq 2$, we have a homemorphism
    \begin{equation}
        \Prim T_n[v_{n-1}^{-1}] \cong \Prim \Ua(\g_n)[v_{n-1}^{-1}] \cong \begin{cases}
       \{(0)\} &  \text{ if } n \text{ is even}, \\ 
       \mb{A}^1 \setminus \{\text{ pt }\} & \text{ if } n \text{ is odd}.
    \end{cases}
    \end{equation}
    Furthermore, when $n$ is odd, then all the primitive ideals of  $T_n[v_{n-1}^{-1}]$ and $\Ua(\g_n)[v_{n-1}^{-1}]$ are maximal two-sided ideals. 
\end{proposition}
\begin{proof}
    By \Cref{lem:idealcor}, we only need to show for $\Ua(\g_n)[v_{n-1}^{-1}]$. Localizing at the normal element $v_{n-1}$, we obtain the following homeomorphism 
    \begin{align*}
        \Spec \Ua(\g_n) \setminus \mathcal{V}(\Ua(\g_n) v_{n-1})  \cong \Spec \Ua(\g_n)[v_{n-1}].
    \end{align*}
    The Dixmier map in \eqref{eq: Dx gn} and \Cref{vn annihilator} induce a homemorphism
    \begin{equation*}
        \{G_n \cdot \bchi \mid \bchi(v_{n-1}) \neq 0 \} \cong \Prim \Ua(\g_n)[v_{n-1}^{-1}]. 
    \end{equation*}
    If $n$ is even, then by \cite[Proposition 7.11]{pham2025orbit}, we have 
    \begin{align*}
        \{G_n \cdot \bchi \mid \bchi(v_{n-1}) \neq 0 \} = \{\eta \in \g_n^* \mid \eta(v_{n-1}) \neq 0\},
    \end{align*}
    which is exactly one coadjoint orbit. By the Dixmier map $\Dx_{\g_n}$ in \eqref{eq: Dx gn}, $\Prim \Ua(\g_n)[v_{n-1}^{-1}]$ is a point. So it equals to the Jacobson radical of $\Ua(\g_n)[v_{n-1}^{-1}]$, which $0$ by \cite{behr1986}. Hence $\Prim \Ua(\g_n)[v_{n-1}^{-1}] = \{(0)\}$. 

    If $n$ is odd, then by \cite[Proposition 7.11]{pham2025orbit}, we have 
    \begin{align*}
        \{G_n \cdot \bchi \mid \bchi(v_{n-1}) \neq 0 \} \cong \mb{A}^1 \setminus \{\text{pt}\}.
    \end{align*}
    Thus, all primitive ideals in this case are maximal. 
\end{proof}

Using \Cref{lem: Jacobson rings 1}, we can characterize prime ideals of $\Ua(\g_n)[v_{n-1}^{-1}]$ as follows.
\begin{proposition}\label{theorem: prime ideals gn}
    The prime spectra $\Spec T_n[v_{n-1}^{-1}]$ and $\Spec \Ua(\g_n)[v_{n-1}^{-1}] $ are homeomorphic. Furthermore,
    \begin{equation}
        \Spec \Ua(\g_n)[v_{n-1}^{-1}] = \begin{cases}
            \Prim \Ua(\g_n) = \{(0)\} & \text{ if } n \text{ is even},\\
            \Prim \Ua(\g_n) \cup \{(0)\} & \text{ if } n \text{ is odd}.
        \end{cases}
    \end{equation}
\end{proposition}
\begin{proof}
    By \Cref{lem:idealcor}, we only need to consider $\Ua(\g_n)[v_{n-1}^{-1}]$. Since $\Ua(\g_n)[v_{n-1}^{-1}]$ is a Jacobson ring by \Cref{lem: Jacobson rings 1}, for a prime ideal $Q$ of $\Ua(\g_n)[v_{n-1}^{-1}]$, we have 
    \begin{align*}
        Q = \bigcap_{P \text{ primitive } \supseteq Q} P.
    \end{align*}
    Let 
    \[S= \{ P \text{ primitive ideal of } \Ua(\g_n)[v_{n-1}^{-1}] \mid P \supseteq Q \}.\] 
    If $|S| = \infty$, then $Q = 0$. If $|S| = 1$, then $Q \in \Prim T_n[v_{n-1}^{-1}]$. Lastly, if $1 < |X| < \infty$, $Q$ is semiprime and not prime, since each primitive ideal of $\Ua(\g_n)[v_{n-1}^{-1}]$ is maximal by \Cref{theorem: primitive ideals gn}, and thus they cannot be contained in one another.
\end{proof}

\subsection{Prime ideals of \texorpdfstring{$B_n$}{Bn} from restriction of prime ideals of \texorpdfstring{$T_n$}{Tn}}\label{subsec: restriction prime}

In this subsection, we completely describe the prime spectrum of $B_n$, in fact, all but one prime ideals of $B_n$ come from $T_n$. In particular, we show that restricting (contracting) gives an ``almost homeomorphism'' between the $\Spec T_n$ and $\Spec B_n$. To that end, we first identify some important ideals of $B_n$ and $T_n$.

\begin{notation}
    For all $n$, let $Q_n^+ = \Psi_n(\Ua(\W{-1})_+)$ be the image of the augmentation ideal of $\Ua(\W{-1})$ under the $n$th orbit homomorphism. Note that $Q_n^+$ is a maximal ideal of $B_n$. 
    
    For all $n \geq m$, recall the projection $\Pi_{n,m}: T_n \surj T_m$ induced by $v_{\geq m} \mapsto 0$. Let 
    \[P^{0, 0}_n  = A_1 \otimes \Ua(\g_n)_+ = \Pi_{n, 1}^{-1}(T_1 v_0) , \quad \text{and} \quad P^{0, 1}_n = \Pi_{n, 1}^{-1}(T_1 (v_0-1))\]
    be prime ideals of $T_n$. Let $Q_n^0 = P^{0, 0}_n \cap B_n$ be its restriction in $B_n$.

    In this subsection, since we will pass between prime ideals of $T_n$ and $B_n$, to make the notation more suggestive, we will often use $P$ for prime ideals of $T_n$ and $Q$ for $B_n$ whenever possible. 
\end{notation}

In \cite{PhamTimmins2025}, $B_n$ and $T_n$ are shown to be birational, in the sense that their rings of fractions coincide $Q(T_n) = Q(B_n)$. We will strengthen that result and prove that restriction of ideals gives an \emph{``almost homeomorphism''} between $\Spec T_n$ and $\Spec B_n$ as follows.

\begin{theorem}\label{theo: Spec almost homeo}
    For all $n$, the map
    \begin{equation}\label{eq: Spec homeo}
        \mathrm{Res}_n: \Spec T_n \to \Spec B_n, \quad P \mapsto P \cap B_n,
    \end{equation}
    satisfies the following properties:
    \begin{enumerate}
        \item $\mathrm{Res}_n$ is continuous;
        \item $\mathrm{Res}_n$ is open;
        \item $\mathrm{Res}_n(\Spec T_n) = \Spec B_n \setminus \{Q_n^+\}$;
        \item $\mathrm{Res}_n^{-1}(Q_n^0) = \{P^{0, 0}_n , \ P^{0, 1}_n \}$;
        \item $\mathrm{Res}_n$ is bijective between $\Spec T_n \setminus \{P^{0, 0}_n, P^{0, 1}_n\}$ and $\Spec B_n \setminus \{Q_n^+, Q_n^0\}$.
    \end{enumerate}
\end{theorem}

Firstly, we use the results in Section \ref{subsec: localization on prime} and Section \ref{subsec: Dx Spec and Prim gn} to show that contraction gives a homeomorphism between prime ideals of $T_n$ not containing $T_n v_{n-1}$ and prime ideal of $B_n$ not containing $T_n v^2_{n-1}$.

\begin{proposition}\label{prop: contraction Tn Bn 1}
    For $n \geq 2$, restriction give a homeomorphism between 
    \[\Spec T_n \backslash \mc{V}(T_n v_{n-1}) \rightarrow \Spec B_n \backslash \mc{V}(T_n v^2_{n-1}).\]
\end{proposition}
\begin{proof}
    By \cite[Theorem 10.20]{goodearl1989introduction}, contraction and extension give homeomorphisms between prime ideals of $T_n[v^{-1}_{n-1}]$ and prime ideals of $T_n$ not containing $v_{n-1}$. We have the following commutative diagram
    \[\begin{tikzcd}
        \Spec B_n \backslash \mc{V}(T_n v^2_{n-1})  \arrow[d, shift right, swap, "\mathrm{ext}_1"]  & \Spec T_n \backslash \mc{V}(T_n v_{n-1})\arrow[ld, swap,"\mathrm{ext}_2"] \arrow[l, "\mathrm{Res}"]  \\ 
        \Spec B_nX_n^{-1} = \Spec T_nX_n^{-1} \arrow[u, swap, "\mathrm{con}_1"] \arrow[ru, shift right, swap, "\mathrm{con}_2"], & 
    \end{tikzcd}\]
    where $\mathrm{ext}_i, \mathrm{con}_i$ are the corresponding extension and contraction of ideals between a ring and its localization. Let $P$ be a prime ideal of $B_n$ not containing $T_n v_{n-1}^2$, which by \Cref{lem: prime two cases} implies that $P \cap X_n \neq 0$. The commutative diagram implies that $PX_n^{-1}$ is a prime ideal of $T_n[v^{-1}_{n-1}]$ and $PX_n^{-1} \cap T_n$ is a prime ideal of $T_n$. By \Cref{prop: extension contraction n}, $(PX_n^{-1} \cap T_n)\cap B_n = P$. Thus, $\mathrm{Res} \circ (\mathrm{con}_2 \circ \mathrm{ext}_1)= \id$. On the other hand, 
    \begin{align*}
        (\mathrm{con}_2 \circ \mathrm{ext}_1) \circ \mathrm{Res} = \mathrm{con}_2 \circ (\mathrm{ext}_1 \circ \mathrm{Res}) = \mathrm{con}_2 \circ \mathrm{ext}_2 = \id,
    \end{align*}
    by \cite[Theorem 10.20]{goodearl1989introduction}. 
    So $\mathrm{Res} = (\mathrm{con}_2 \circ \mathrm{ext}_1)^{-1} = \mathrm{con}_1 \circ \mathrm{ext}_2$, which is a homeomorphism.
\end{proof}

We next show that all prime ideals of $B_n$ are completely prime. This result demonstrates the powerful combination of considering homomorphic images and (noncommutative) localization, which is a common theme for the rest of the paper. 

\prcprime*
\begin{proof}
    We will prove the statement by induction. The case $B_0$ and $B_1$ are \Cref{prop:B0ideals} and \Cref{Spec B1}. Assuming induction hypothesis for $k \geq 1$. Since $T_{k+1} v_{k}^2 = \ker \Pi_{k+1, k} \cap B_{k+1}$, the vanishing locus $\mc{V}(T_{k+1} v_{k}^2)$ consists of only completely prime ideals. Thus, we only need to show that restriction of a prime ideal $P$ of $T_{k+1}$ not containing $v_{k}$ to $B_{k+1}$ is a completely prime ideal. Since $\g_{k+1}$ is a solvable Lie algebra, by \cite[Theorem 3.7.2]{dixmier1996enveloping}, every prime ideal of $\Ua(\g_{k+1})$ is completely prime. Thus, every prime ideal of $A_1 \otimes \Ua(\g_{k+1})$ is completely prime as well, since Ore extension of domains are domains. Thus $T_{k+1}/P$ is a domain, and so is the subring $B_{k+1}/(P \cap B_{k+1}) \cong (B_{k+1} + P)/P$.
\end{proof}

The restriction map $\mathrm{Res}_n: \Spec T_n \to \Spec B_n$ in \eqref{eq: Spec homeo} is ``almost surjective'', in the sense that its image misses exactly one-point. The proof again is a combination of induction on homomorphic image and localization.
\begin{lemma}\label{lem: image Res}
    The image $\mathrm{Res}_n(\Spec T_n)$ of the restriction map \eqref{eq: Spec homeo} equals $\Spec B_n \setminus \{Q_n^+\}$.
\end{lemma}
\begin{proof}
    For $n=0$ and $n=1$, the statement follows from \Cref{prop:B0ideals} and \Cref{Spec B1}. Assuming induction hypothesis, by \Cref{lem: prime two cases}, we need to consider if $Q$ contains $T_n v_{n-1}^2$ or $Q \cap X_n = \emptyset$. If $Q$ contains $T_n v_{n-1}^2$, then $Q$ is the preimage of a prime ideal $Q'$ of $B_{n-1}$. By induction hypothesis, there exists a prime ideal $P'$ of $T_{n-1}$ such that
    \begin{align*}
        Q = \Pi_{n,n-1}^{-1}(Q') = \Pi_{n,n-1}^{-1}(P' \cap B_{n-1}) = \Pi_{n,n-1}^{-1}(P') \cap B_n, 
    \end{align*}
    and $\Pi_{n,n-1}^{-1}(P')$ is a prime ideal of $T_n$ containing $T_n v_{n-1}$. On the other hand, if $Q \cap X_n = \emptyset$, the statement follows from \Cref{prop: contraction Tn Bn 1}. 
\end{proof}

We now prove a series of lemmata that characterizes the topological properties of the map $\mathrm{Res}_n$. 

\begin{lemma}\label{lem: Res1 open}
    The homomorphism $\mathrm{Res}_1: \Spec T_1 \to \Spec B_1$ is open.
\end{lemma}
\begin{proof}
    By \Cref{lem:idealcor}, $\Spec T_1 \cong \Spec \kk[v_0]= \{0, T_1 (v_0-\alpha) \mid \alpha \in \kk\}$. Let $P_\alpha = T_1 (v_0-\alpha)$. On the other hand, $\Spec B_1 =\{0, Q^0_1 = T_1 v_0 (v_0-1), Q^+_1, P_\alpha \cap B_n\}$ by \Cref{Spec B1}. Since $Q^+_1$ and $P_\alpha \cap B_n$ are maximal ideals of $B_n$ for all $\alpha \neq 0, 1$, the subsets $\{Q^+_1\}$, $\{P_\alpha \cap B_n\}$ are closed in $\Spec B_n$. Furthermore, $\mc{V}(Q^0_1) = \{Q^0_1, Q^+_1\}$. Thus, it is easy to check that $\mathrm{Res}_1$ maps open sets of $\Spec T_1$ to open sets of $\Spec B_1$. 
\end{proof}

As $B_n$ and $T_n$ satisfy ACC on two-sided ideals by \Cref{theo: Bn Cn ACC}, their spectra are Noetherian topological spaces. Recall that for Noetherian topological spaces, we can define a \emph{constructible subset} as a finite union of locally closed subsets. The following lemma is well-known.
\begin{lemma}
    Let $Z$ be a Noetherian topological space. Let $E \subseteq Z$ be a subset. The following are equivalent:
    \begin{enumerate}
        \item $E$ is constructible in $Z$, i.e., a finite union of locally closed set, 
        \item for every irreducible closed $Y \subset Z$ the intersection $E \cap Y$ either contains a nonempty open of $Y$ or is not dense in $Y$.\qed
    \end{enumerate}
\end{lemma}

The following proposition could be seen as a form of Chevalley's theorem applied in our noncommutative settings of $T_n$ and $B_n$.

\begin{proposition}\label{prop: constructible}
    For all $n \in \NN$, the homomorphism $\mathrm{Res}_n: \Spec T_n \to \Spec B_n$ sends open sets to constructible sets.
\end{proposition}

\begin{proof}
    We will prove the statement by induction on $n$. The base case $n=1$ follows from \Cref{lem: Res1 open}. Assuming induction hypothesis, let $I$ be an ideal in $T_n$. Let $D(I)$ denote the open set $\Spec T_n \setminus \mc{V}(I)$. Let $\mc{P} = \Spec T_n \backslash \mc{V}(T_n v_{n-1})$ and $\mc{Q} = \Spec B_n \backslash \mc{V}(T_n v^2_{n-1})$. It suffices to prove that both $\mathrm{Res}_n(D(I) \cap \mc{P})$ and $\mathrm{Res}_n(D(I) \cap T_n v^2_{n-1})$ are constructible in $\Spec B_n$. By \Cref{prop: contraction Tn Bn 1}, 
    \[\mathrm{Res}_n: \Spec T_n \backslash \mc{V}(T_n v_{n-1}) \rightarrow \Spec B_n \backslash \mc{V}(T_n v^2_{n-1})\]
    is a homeomorphism, so $\mathrm{Res}_n(D(I) \cap \mc{P})$ is an open subset of $\mc{Q}$ since $D(I) \cap \mc{P}$ is open in $\mc{P}$. Hence $\mathrm{Res}_n(D(I) \cap \mc{P})$ is open, and thus constructible in $\Spec B_n$.  
    
    On the other hand, we have a commutative diagram 
    \[
    \begin{tikzcd}
        \Spec T_{n-1} \arrow[d, "\Pi_{n, n-1}^\sharp"] \arrow[r, "\mathrm{Res}_{n-1}"] & \Spec B_{n-1} \arrow[d, "\xi^\sharp_{n, n-1}"]\\
        \Spec T_n \arrow[r, "\mathrm{Res}_n"]& \Spec B_n,
    \end{tikzcd}
    \]
    where we denote the map $\xi_{n, n-1}: B_n \surj B_{n-1}$ in this diagram to avoid confusion. Then $Z:= (\Pi^{\sharp}_{n, n-1})^{-1} \big ( D(I) \cap \mc{V}(T_n v_{n-1}) \big )$ is open in $\Spec T_{n-1}$, and is in fact equals $D(\Pi_{n, n-1}(I))$. By induction hypothesis, $\mathrm{Res}_{n-1} (Z)$ is constructible in $\Spec B_{n-1}$. Since $\xi_{n, n-1}$ is surjective with kernel $T_n v_{n-1}^2$, it induces a homeomorphism from $\Spec B_{n-1}$ onto the closed set $\mc{V}(T_n v_{n-1}^2)$ of $B_n$. Thus, by the commutative diagram, $\mathrm{Res}_n$ is constructible in $\mc{V}(T_n v_{n-1}^2) \subseteq \Spec B_n$, and hence is constructible in $\Spec B_n$.
\end{proof}

Motivated by algebraic geometry, we define the following notion of ``going down'' property.
\begin{definition}
    We call a ring homomorphism $\varphi: R \to S$ or the map $\varphi^\sharp: \Spec S \to \Spec R$ satisfies the \emph{going down property} if for all prime ideals $Q \subseteq Q'$ of $R$, and $P' \in \Spec S$ lying over $Q'$, i.e., $\varphi^\sharp(P') = Q'$, then there exists $P \subseteq P'$ lying over $Q$.
\end{definition}

\begin{proposition}\label{prop: Res going down}
    The homomorphism $\mathrm{Res}_n: \Spec T_n \to \Spec B_n$ satisfies the going down property.
\end{proposition}
\begin{proof}
    We will prove the statement by induction. The base case $n=1$ follows can be easily checked from \Cref{Spec B1}. 
    
    Assuming induction hypothesis, let $Q \subseteq Q'$ be nonzero prime ideals of $B_n$ and $P'$ is a prime lying over $Q'$. If $Q$ contains $T_n v_{n-1}^2$, then $Q'$ also contains $T_n v_{n-1}^2$ and hence $P'$ contains $T_n v_{n-1}$. Thus, $\Pi_{n, n-1}(Q) \subseteq \Pi_{n, n-1}(Q')$ are prime ideals of $B_{n-1}$, moreover $\Pi_{n, n-1}(P')$ is a prime ideal of $T_n$. By induction hypothesis, there exists a prime ideal $J$ of $T_{n-1}$ lying over $\Pi_{n, n-1}(Q)$ such that $J \subseteq \Pi_{n, n-1}(P')$. Hence $\Pi^{-1}_{n, n-1}(J)$ is a prime ideal of $T_n$ lying over $Q$ and $\Pi^{-1}_{n, n-1}(J) \subseteq P'$.

    Otherwise, if $Q \cap X_n = \emptyset$, then $Q \in \Spec B_n X_n^{-1} \cong \Spec \Ua(\g_n[v_{n-1}^{-1}])$. Since $Q$ is nonzero, by \Cref{theorem: prime ideals gn} and \Cref{theorem: primitive ideals gn}, $Q$ is a maximal ideal of $B_n$, thus $Q' = Q$ and $P' = P$.
\end{proof}

    The following definitions are motivated by algebraic geometry and is purely topological.
\begin{definition}
    For two points $P,P'$ in topological space $Y$, then we say $P$ is a \emph{specialization} of $P'$, or $P'$ is a \emph{generalization} of $P$ if $P \in \overline{P'}$.

    A subset $T \subseteq Y$ is called \emph{stable under generalization} if for all $P \in T$ and every generalization $P \subseteq P'$ we have $P' \in T$. The notion of being \emph{stable under specialization} is defined analogously.
\end{definition}

The following lemma is well-known.
\begin{lemma}\label{lem: generalization open}
    Let $Y$ be Noetherian sober topological space, and $E$ is a subset of $Y$.
    \begin{itemize}
        \item If $E$ is constructible and stable under specialization, then $E$ is closed.
        \item If $E$ is constructible and stable under generalization, then $E$ is open.
    \end{itemize}
    \qed
\end{lemma}
Finally, we put everything together to prove the main theorem of this section 
\begin{proof}[Proof of \Cref{theo: Spec almost homeo}]
    
(1) is obvious since the map $\mathrm{Res}_n$ is induced from the inclusion $B_n \to T_n$, while (3) is \Cref{lem: image Res}.

We now show (2) that $\mathrm{Res}_n$ is open. Let $D(I)= \Spec T_n \setminus \mc{V}(I)$ be an open set of $\Spec T_n$. By \Cref{prop: constructible}, $\mathrm{Res}_n(D(I))$ is constructible. By \Cref{prop: Res going down}, $\mathrm{Res}_n$ satisfies going down, which implies that $\mathrm{Res}_n(D(I))$ is stable under generalization. As $B_n$ satisfies ACC on two-sided ideals, $\Spec B_n$ is a Noetherian topological space, which is automatically sober by definition of Zariski topology, thus by \Cref{lem: generalization open}, $\mathrm{Res}_n(D(I))$ is open in $\Spec B_n$. 

Since $T_1 v_0 \cap B_1 = T_1 v_0 (v_0-1) = T_1 (v_0 -1) \cap B_1$, we have $P_{n}^{0, 0}, P^{0, 1}_n \in \mathrm{Res}_n^{-1}(Q^0_n)$. So it suffices to prove (5), and in fact, suffices to show that $\mathrm{Res}$ is injective on $\Spec T_n \setminus \{P_{n}^{0, 0}, P^{0, 1}_n\}$. For $n=0$ and $n=1$, this again follows from \Cref{prop:B0ideals} and \Cref{Spec B1}. If $P$ is a prime ideal of $B_n$ containing $T_{n}v^2_{n-1}$, then since ideals of $B_n$ containing $\ker \Pi_{n, n-1} \cap B_n$ corresponds one-to-one to ideals of $B_{n-1}$, induction hypothesis implies the statement. On the other hand, if $P$ does not contain $T_{n}v^2_{n-1}$, then $P \cap X_n = \emptyset$ and thus the statement follows from \Cref{prop: contraction Tn Bn 1}.     
\end{proof}

\section{The Dixmier-Moeglin equivalence and the primitive spectrum of \texorpdfstring{$B_n$}{Bn}}\label{sec: DME and Prim Bn}

In this section, we completely characterize primitive ideals of $B_n$. The key step is showing that $B_n$ satisfies the Dixmier Moeglin equivalence (DME), which gives two additional characteristics for primitive ideals. Lastly, we present an interesting application of our classification results to the Dixmier map for the Witt algebra constructed in \cite{pham2025orbit}. In particular, we make significant progress towards understanding injectivity, surjectivity, continuity and openness of the Dixmier map for the Witt algebra in \cite{pham2025orbit}. 

\subsection{Primitive ideals of \texorpdfstring{$B_n$}{Bn}}\label{subsec: Prim Bn}

In this subsection, we first show that $B_n$ are Jacobson rings and satisfies the Nullstellensatz over $\kk$. This gives us the classical direction of the Dixmier-Moeglin equivalence (DME): a prime ideal is locally closed (LC) in the Zariski topology implies it is primitive (P) implies it is rational (R), i.e., 
\[\text{(LC)} \Rightarrow \text{(P)} \Rightarrow \text{(R)}.\] 
We then prove that if a prime of $B_n$ is rational, it is also locally closed, which completes the chain and shows that $B_n$ satisfies the DME for all $n$. We then classify all primitive ideals of $B_n$, showing that all but one come from restricting primitive ideals of $T_n$.

Firstly, we show that it makes sense to talk about rationality for prime ideals of $B_n$. 
\begin{lemma}\label{lem: rational makes sense}
    Let $P$ be a prime ideal of $B_n$. Then $B_n/P$ is a Goldie ring.
\end{lemma}
\begin{proof}
    By \Cref{prime cprime}, $P$ is completely prime. Moreover, by \cite[Proposition 3.15]{krause2000growth}, 
    \[\GKdim B_n/P \leq \GKdim B_n -1 = n+1,\]
    where the last equation follows from \cite[Corollary 1.2]{PhamTimmins2025}. Thus $B_n/P$ is a domain of finite GK-dimension, and hence, it is Goldie by \cite[Theorem 4.15]{krause2000growth}.
\end{proof}

In the following lemma, we show that $B_n$ are Jacobson rings that satisfies the Nullstellensatz over $\kk$. The strategy is similar to \Cref{lem: Jacobson rings 1} except that $B_n$ is not Noetherian.
\begin{lemma}\label{lem: Jacobson and Nullstellensatz}
    For all $n$, $B_n$ are Jacobson rings. Moreover, $\Ua(W), \Ua(\W{-1}), \Ua(\Vir)$ and $B_n$ all satisfy the Nullstellensatz over $\kk$. 
\end{lemma}
\begin{proof}
    By \cite[Corollary 9.1.8]{mcconnell2001noncommutative}, $\Ua(W), \Ua(\W{-1}), \Ua(\Vir)$ and $B_n$ satisfy the Nullstellensatz over $\kk$. 

    Let $I$ be an ideal of $B_n$. Since $B_n/I$ is a finitely generated $\kk$-algebra, $\dim_\kk B_n/I$ is countable and thus by \cite[Corollary 4]{Amitsur56}, the Jacobson radical $J(B_n/I)$ is nil, and thus $B_n$ satisfies the radical property. 

    Since all prime factor rings of $B_n$ are right Goldie by \Cref{lem: rational makes sense}, $B_n$ is a Jacobson ring by \cite[Lemma 9.1.2 (ii)]{mcconnell2001noncommutative}.      
\end{proof}

Thus, the classical direction of the DME holds for $B_n$.
\begin{proposition}\label{prop: classical DME for Bn}
    Let $P$ be a prime ideal of $B_n$. If $P$ is locally closed in the Zariski topology, then $P$ is primitive. Moreover, if $P$ is primitive, then $P$ is a rational ideal, i.e. $Z(Q(B_n/P)) = \kk$. In other words,  
    \begin{equation}
        \text{(LC)} \Rightarrow \text{(P)} \Rightarrow \text{(R)}. 
    \end{equation}
\end{proposition}
\begin{proof}
    This follows from \Cref{lem: rational makes sense}, \Cref{lem: Jacobson and Nullstellensatz} and \Cref{prop: classical DME direction}.
\end{proof}

For all $i \in \NN$, let $D_i(\kk) = Q(A_i(\kk))$ denote the division ring of fraction of the $i$th Weyl algebra. Furthermore for $j \in \NN$, let $A_i(\kk[z_1, \dots, z_j]) = A_i(\kk)[z_1, \dots, z_j]$ be the Weyl algebra adjoined by $j$ central indeterminates. We then denote its ring of fractions by $D_{i, j}(\kk)$ , i.e., 
\[D_{i, j}(\kk) =  Q(A_i(\kk[z_1, \dots, z_j])) = D_i(k) (z_1, \dots, z_j).\]

\begin{lemma}\label{lemma: center quotient field}
    If $n$ is even, then $Z(Q(B_n)) = \kk$. If $n$ is odd, then $Z(Q(B_n)) = \kk(z)$.
\end{lemma}
\begin{proof}
    Firstly, since $\g_n$ is a finite-dimensional solvable Lie algebra, by \cite{joseph1974ProofOT}, we have
    \begin{equation}
        Q(\Ua(\g)) = D_{m, s}(\kk),
    \end{equation}
    where $m = \left\lfloor \frac{n}{2}\right\rfloor$ and $s = n-2m$. Then 
    \begin{align*}
        Q(T_n) = Q(A_1 \otimes \Ua(\g_n)) = Q(A_1 \otimes Q(\Ua(\g_n))) = Q(A_1 \otimes D_{m, s}(\kk)) = D_{m+1, s}(\kk).
    \end{align*}
    Thus the lemma follows as $Q(B_n) = Q(T_n)$ by \cite[Corollary 1.2]{PhamTimmins2025}.
\end{proof}

With that, we are ready to show that all the homomorphic images $B_n$ satisfies the DME. 
\BnDME*
\begin{proof}
    By \Cref{prop: classical DME for Bn}, it suffices to check that a rational prime ideal of $B_n$ is locally closed in the Zariski topology. We will show this statement by induction. For $n=0$, by \Cref{prop:B0ideals}, the only prime ideals of $B_0$ are $0$ and $A_1\del$, which are both rational and locally closed. If $n =1$ by \Cref{Spec B1}, we have the prime ideals of $B_1$ are 
    $\Psi_1(\Ua(\W{-1})_+),\  T_1(v_0-v_0^2),\  T_1 (v_0 - \alpha) \cap B_1$ where $\alpha \neq 0, 1$, which are all rational. The ideals $\Psi_1(\Ua(\W{-1})_+)$ and $T_1 (v_0 - \alpha) \cap B_1$ are both maximal, while $T_1(v_0-v_0^2) \subsetneqq \Psi_1(\Ua(\W{-1})_+)$, and thus is also locally closed.
    
    Set $\Pi = \Pi_{n, n-1}$. Assuming induction hypothesis that $B_{n-1}$ satisfies the DME. Let $P$ be a rational prime ideal of $B_n$, i.e. $Z(Q(B_n /P))= \kk$. 
    
    Suppose $P$ contains $T_n v_{n-1}^2$, then $P$ is the preimage of a prime ideal $P'$ of $B_{n-1}$. In particular, we have 
    \begin{align*}
        B_{n-1}/P' \cong B_n /P \Rightarrow Z(Q(B_{n-1}/P')) \cong Z(Q(B_n /P))= \kk, 
    \end{align*}
    thus $P'$ is rational in $B_{n-1}$. By induction hypothesis, since $B_{n-1}$ satisfies the DME, $P'$ is locally closed, i.e. 
    \begin{align*}
        P' \subsetneqq \bigcap_{I' \text{ prime } \supsetneqq P'} I'. 
    \end{align*} 
    Suppose $I$ is an ideal of $B_n$ containing $P$, then $I$ contains $T_{n-1}v_{n-1}^2$ and thus is a preimage of a prime ideal of $B_{n-1}$. So 
    \begin{align*}
        \bigcap_{I \text{ prime } \supsetneqq P} I =  \bigcap_{I' \text{ prime } \supsetneqq P'} \Pi^{-1}(I') = \Pi^{-1} \big (\bigcap_{I' \text{ prime } \supseteq P'} I' \big ) \supsetneqq \Pi^{-1}(P') = P,
    \end{align*} 
    by induction hypothesis, and thus $P$ is locally closed in $B_n$. 

    On the other hand, suppose $P$ does not contain $T_n v_{n-1}^2$. If $n$ is even, by \Cref{theorem: prime ideals gn}, then $0$ is the only prime ideal of $B_n$ not containing $T_n v_{n-1}^2$. In this case $0$ is primitive and thus rational. We now show that it is locally closed. We have
    \begin{equation*}
        \bigcap_{I \text{ prime } \supsetneqq 0} I = \bigcap_{I \text{ prime } \supseteq T_n v_{n-1}^2} I = T_n v_{n-1}^2 \supsetneqq 0.
    \end{equation*}
    Lastly, if $n$ is odd, then the zero ideal $0$ is not rational by \Cref{lemma: center quotient field}. Otherwise, by \Cref{theorem: prime ideals gn}, then $P \in \Prim T_n[v_{n-1}^{-1}]$, which is maximal by \Cref{theorem: primitive ideals gn} and thus is locally closed.
\end{proof}

Thus, \Cref{theo: SpecalmosthomeoPrim} follows from \Cref{theo: Spec almost homeo} and the fact that both $T_n$ and $B_n$ satisfy the DME in \cite[Corollary 3.6.10]{phamthesis} and \Cref{theo: BnDME}. 

\begin{theorem}\label{theo: SpecalmosthomeoPrim}
    For all $n$, the map $\mathrm{Res}: \Spec T_n \to \Spec B_n$ in \Cref{theo: Spec almost homeo} gives a well-defined continuous and open map 
    \begin{equation}
        \mathrm{Res}: \Prim T_n \to \Prim B_n.
    \end{equation}
    Moreover, $\mathrm{Res}_n$ is bijective between $\Prim T_n \setminus \{P^{0, 0}_n, P^{0, 1}_n\}$ and $\Prim B_n \setminus \{Q_n^+, Q_n^0\}$.
\end{theorem}
\begin{proof}
    Let $P$ be a primitive ideal of $T_n$ that is not $\Pi_{n, 1}^{-1}(T_1 v_0)$ or $\Pi_{n, 1}^{-1}(T_1 (v_0-1))$.  Since $T_n$ satisfies the DME by \cite[Corollary 3.6.10]{phamthesis}, $P$ is locally closed in $\Spec T_n$. Let $P'$ be any prime ideal of $T_n$ that strictly contains $P$. By \Cref{theo: Spec almost homeo}.4, $P' \cap B_n$ is a prime ideal of $B_n$ that strictly contains $P \cap B_n$, so $P \cap B_n$ is locally closed in $B_n$. Since $B_n$ also satisfies the DME by \Cref{theo: BnDME}, $P$ is primitive. Note that 
    \[\Pi_{n, 1}^{-1}(T_1 v_0) \cap B_n = \Pi_{n, 1}^{-1}(T_1 (v_0-1)) \cap B_n = Q_n^0.\]
    Since $B_n/Q_n^0 = B_0$, it is easy to see that $Q_n^0$ is a rational prime ideal of $B_n$. Thus since $B_n$ satisfies the DME by \Cref{theo: BnDME}, $Q_n^0$ is a primitive ideal of $B_n$. 
    
    Let $P$ be a primitive ideal of $B_n$ not equals $Q_n^+$ or $Q_n^0$. Then $P$ is a prime ideal of $B_n$ coming from restricting a unique prime ideal $P'$ of $T_n$ by \Cref{theo: Spec almost homeo}. Since $B_n$ satisfies the DME by \Cref{theo: BnDME}, $P$ is locally closed in $\Spec B_n$. Any prime ideal of $B_n$ not equals $Q_n^+$ or $Q_n^0$ comes from restricting a unique prime ideal $T_n$ by \Cref{theo: Spec almost homeo}, so $P'$ is locally closed in $\Spec T_n$. As $T_n$ satisfies the DME, $P'$ is primitive and the theorem follows.
\end{proof}

\subsection{Applications to the Dixmier map}\label{subsec: app on Dix}

Lastly, in this subsection, we show several applications of our results, especially Section \ref{subsec: Prim Bn}, to the Dixmier map for the Witt and Virasoro algebra in \cite{pham2025orbit}. Firstly, we completely characterize the ``injectivity'' of the Dixmier map for $\W{-1}$ restricted to the fundamental building blocks, known as one-point local functions. Our results also immediately implies that all primitive ideals of $B_n$ are in the image of the Dixmier map, which gives significant evidence for the conjecture that the Dixmier map is surjective. Lastly, we completely characterize the inclusion conditions for these primitive ideals of one-point local functions. 

First, we recall the central notions from the Dixmier map in \cite{pham2025orbit}. Since $\W{-1}$ has countable dimension and $\kk$ is uncountable, the symmetric algebra $\Sa(\W{-1})$ satisfies the extended Nullstellensatz \cite[Corollary
9.1.8]{mcconnell2001noncommutative}. In particular, maximal ideals of $\Sa(\W{-1})$ can be canonically identified with $\W{-1}^*$: 
\[\chi \in \W{-1}^* \leftrightarrow \mf{m}_\chi:= \ker \ev_{\chi},\]
where $\ev_{\chi}$ the evaluation homomorphism $\Sa(\W{-1}) \to \kk$ defined by $\ev_{\chi} = f(\chi)$. A Poisson primitive ideal of $\Sa(\W{-1})$ is defined as  the largest Poisson ideal contained in a maximal ideal of $\Sa(\W{-1})$. Thus, every Poisson primitive ideal of $\Sa(\W{-1})$ are of the form $P(\chi)$ for some $\chi \in \W{-1}^*$ and we denote the Poisson primitive spectrum of $\Sa(\W{-1})$ by $\Pprim \Sa(\W{-1})$. 

By \cite[Theorem 3.1]{petukhov2022poisson}, the Poisson primitive ideal $P(\chi)$ is nonzero if and only if $\chi$ is a \emph{local function} on $\W{-1}$, which is a finite sum of many one-point local functions below. For $x, \alpha_0, \dots, \alpha_n \in \kk$ with $\alpha_n \neq 0$, we define a \emph{one-point local function} $\chi_{x; \alpha_0, \dots, \alpha_n}$ as the linear functional
\begin{equation}
    \chi_{x; \alpha_0, \dots, \alpha_n}: \W{-1} \to \kk, \quad f\del \mapsto \alpha_0 f(x) + \dots + \alpha_n f^{(n)}(x),
\end{equation} 
where $f^{(i)}$ denotes the $i$th-derivative of $f$. 

One of the main results from \cite{pham2025orbit} and \cite{Ushirobira1998} is that using the famous Kirillov orbit method \cite{Kirillov2004LecturesOT}, we can construct primitive ideals of $\Ua(\W{-1})$ from local functions as follows. For all local functions $\chi$ on $\W{-1}$, there exists a canonical polarization $\p_\chi$ of $\W{-1}$, which is a maximal isotropic subspace of $\W{-1}$ with respect to the induced bilinear form, and that is also a finite-codimensional subalgebra of $\W{-1}$. The local function $\chi$ then defines a one-dimensional representation $\kk_\chi$ of $\p_\chi$, and we induce it to a representation $M_\chi = \Ua(\W{-1}) \otimes_{\Ua(\p_\chi)} \kk_\chi$ of $\W{-1}$. Then, Theorem 1.3 of \cite{pham2025orbit} shows that the annihilator $Q_\chi = \Ann_{\Ua(\W{-1})} M_\chi$ is always a primitive ideal. Moreover, Theorem 1.1 of \cite{pham2025orbit} shows that if $P(\chi) = P(\eta)$, then $Q_\chi = Q_\eta$, which gives a well-defined Dixmier map 
\[\widetilde{\Dx}: \Pprim \Sa(\W{-1}) \to \Prim \Ua(\W{-1}), \quad \text{induced by} \quad P(\chi) \mapsto Q_\chi.\]

There are various open questions regarding properties of the Dixmier map, including its injectivity, continuity and openness. Most importantly, answering \cite[Conjecture 8.6]{pham2025orbit} will make significant progress towards the classification of irreducible representations of the Witt and Virasoro algebras.

We first completely characterize the fiber of the Dixmier map for one-point local function. As expected in \cite{pham2025orbit}, in that case, the Dixmier map is injective except for the case $\chi = \chi_{x; \alpha, 1}$ and $\eta = \chi_{y; \beta}$ where $x, y, \alpha \in \kk$ and $\beta \in \kk^\times$.

\begin{proposition}\label{prop: Dx inject}
    Let $\chi$ and $\eta$ be one-point local functions on $\W{-1}$. Then $Q_\chi  = Q_\eta$ if and only if $P(\chi) = P(\eta)$ or $\chi = \chi_{x; \alpha, 1}$ and $\eta = \chi_{y; \beta , 0}$ for some $\alpha \in \kk$ and $\beta \in \kk^\times$.
\end{proposition}
\begin{proof}
    We first show the forward direction. Assume $Q_\chi = Q_\eta$ and $\chi \neq \chi_{x; \alpha, 1}$ for some $\alpha \in \kk$. Then $\eta \neq \chi_{x; \beta , 0}$ for all $\beta \in \kk^\times$ as $Q_{\chi_{x; \beta , 0}} = Q_{\chi_{x; \alpha, 1}}$ by \cite[Proposition 4.28]{pham2025orbit}. We also assume $\chi$ is not a zero local function, and thus $\eta$ is also not zero. WLOG, suppose that the order of $\chi$ is $n$ and is larger than the order of $\eta$. Then by \cite[Proposition 3.4.2]{phamthesis} or \cite[Proposition 6.8]{pham2025orbit}, there exists $\bchi$ and $\overline{\eta} \in \g^*_n$ such that  
    \begin{align*}
        Q_\chi = \Psi_{n}^{-1}(A_1 \otimes Q'_{\bchi}) = \Psi_{n}^{-1}(A_1 \otimes Q'_{\overline{\eta}}) =  Q_\eta.
    \end{align*}
    Thus $(A_1 \otimes Q'_{\bchi}) \cap B_n = A_1 \otimes Q'_{\overline{\eta}} \cap B_n$. By assumption, this ideal could not be $Q_n^+$ or $Q_n^0$, so they corresponds one-to-one to prime ideal of $T_n$ by \Cref{theo: Spec almost homeo}. Thus $Q'_{\bchi} = Q'_{\overline{\eta}}$. By \cite[Theorem 3.2.11]{phamthesis} or \cite[Theorem 2.12]{pham2025orbit}, $G_n \cdot \bchi = G_n \cdot \overline{\eta}$. Lastly, \cite[Theorem 3.5.17]{phamthesis} or \cite[Theorem 1.5]{pham2025orbit} implies that $\orbit{\chi} = \orbit{\eta}$. 

    The backward direction is obvious so we omit the proof.
\end{proof}

\Cref{theo: SpecalmosthomeoPrim} thus establishes a significant step towards the conjecture on the surjectivity of the Dixmier map $\widetilde{\Dx}_{\W{-1}}$. By \cite[Theorem 1.4]{PhamTimmins2025}, $\ker \Psi_n$ is generated by the differentiators $\Omega^{m}_{k, s}$, which has the following closed from 
\[\Omega^{m}_{k, s} = \sum_{i=0}^{m}(-1)^i \binom{m}{i}e_{k-i}e_{s+i}.\]

\begin{proposition}\label{prop: Dx surjective}
    Any primitive ideal of $\Ua(\W{-1})$ which contains any nonzero differentiator $\Omega^{m}_{k, s}$ is in the image of the Dixmier map $\widetilde{\Dx}_{\W{-1}}: \Pprim \Sa(\W{-1}) \to \Prim \Ua(\W{-1})$.
\end{proposition}

\begin{proof}
    This follows since the kernel of $\Psi_n$ is generated by some nonzero $\Omega^{m}_{k, s}$, from \cite[Theorem 1.4]{PhamTimmins2025} and \Cref{theo: SpecalmosthomeoPrim}. 
\end{proof}

Similarly, we also provide significant evidence for \cite[Conjecture 8.7]{pham2025orbit} that all prime ideals of $\Ua(\W{-1})$ are completely prime
\begin{corollary}
    Any prime ideal of $\Ua(\W{-1})$ containing any nonzero differentiator $\Omega^{m}_{k, s}$ is completely prime.
\end{corollary}
\begin{proof}
    This follows since the kernel of $\Psi_n$ is generated by some nonzero $\Omega^{m}_{k, s}$, from \cite[Theorem 1.4]{PhamTimmins2025} and \Cref{prime cprime}. 
\end{proof}
Lastly, we prove \cite[Conjecture 8.4]{pham2025orbit} that for one-point local functions $\chi$ and $\eta$ on $\W{-1}$, the inclusion $Q_\chi \subsetneqq Q_\eta$ is equivalent to the inclusion $P(\chi) \subsetneqq P(\eta)$. 
\begin{proposition}\label{prop: inverse inclusion}
    Let $\chi, \eta$ be one-point local functions on $\W{-1}$. Then 
    \begin{align*}
        P(\chi) \subsetneqq P(\eta) \Leftrightarrow Q_\chi \subsetneqq Q_\eta.
    \end{align*}
\end{proposition}
\begin{proof}
    \cite[Proposition 8.2]{pham2025orbit} says that if $\chi$ and $\eta$ are one-point local functions, then the inclusion $P(\chi) \subseteq P(\eta)$ implies $Q_\chi \subseteq Q_\eta$. Together with \Cref{prop: Dx inject}, that gives us the forward direction. It thus suffices to show the inverse direction.

    Let $n$ be the maximal order of $\chi$ and $\eta$. Then by \cite[Proposition 3.4.2]{phamthesis} or \cite[Proposition 6.8]{pham2025orbit}, there exists $\bchi$ and $\overline{\eta} \in \g^*_n$ such that  
    \begin{align*}
        Q_\chi = \Psi_{n}^{-1}(A_1 \otimes Q'_{\bchi}) \subsetneqq \Psi_{n}^{-1}(A_1 \otimes Q'_{\overline{\eta}}) =  Q_\eta.
    \end{align*}
    which implies that $(A_1 \otimes Q'_{\bchi}) \cap B_n \subsetneqq (A_1 \otimes Q'_{\overline{\eta}})\cap B_n$, neither of which is $Q_n^+$ by assumption. By \Cref{theo: Spec almost homeo}, we have 
    \[A_1 \otimes Q'_{\bchi} \subsetneqq A_1 \otimes Q'_{\overline{\eta}}.\] 
    Thus \Cref{lem:idealcor}, $Q'_{\bchi} \subsetneqq Q'_{\overline{\eta}}$. The fact that the Dixmier map $\Dx_{\g_n}: \g^*_n/G_n \to \Prim \Ua(\g_n)$ is open is equivalent to $G_n \cdot \overline{\eta} \subsetneqq G_n \cdot {\bchi}$ by \cite[(3.3.3) and Lemma 6.9]{mathieu1991bicon}. Thus by \cite[Proposition 3.6.12]{phamthesis}, $P(\chi) \subsetneqq P(\eta)$.
\end{proof}

\bibliographystyle{amsalpha}
\bibliography{../refs.bib}

\providecommand{\bysame}{\leavevmode\hbox to3em{\hrulefill}\thinspace}
\providecommand{\MR}{\relax\ifhmode\unskip\space\fi MR }
\providecommand{\MRhref}[2]{%
  \href{http://www.ams.org/mathscinet-getitem?mr=#1}{#2}
}
\providecommand{\href}[2]{#2}
\begin{thebibliography}{Moe80}

\bibitem[AM25]{andruskiewitsch2025noetherian}
Nicol\'as Andruskiewitsch and Olivier Mathieu, \emph{Noetherian enveloping
  algebras of simple graded {L}ie algebras}, J. Math. Soc. Japan \textbf{77}
  (2025), no.~4, 1233--1247.

\bibitem[Ami56]{Amitsur56}
Avraham~S. Amitsur, \emph{Algebras over infinite fields}, Proc. Amer. Math.
  Soc. \textbf{7} (1956), 35--48.

\bibitem[Beh86]{behr1986}
Erazm~J. Behr, \emph{Jacobson radical of filtered algebras}, Proc. Amer. Math.
  Soc. \textbf{98} (1986), no.~4, 545--546.

\bibitem[Bel19]{BellDME2019}
Jason Bell, \emph{On the importance of being primitive}, Rev. Colombiana Mat.
  \textbf{53} (2019), 87--112.

\bibitem[BG02]{BrownGoodearBook}
Ken~A. Brown and Ken~R. Goodearl, \emph{Lectures on algebraic quantum groups},
  Advanced Courses in Mathematics. CRM Barcelona, Birkh\"auser Verlag, Basel,
  2002.

\bibitem[BG03]{brown2002Poisson}
Ken~A. Brown and Iain Gordon, \emph{Poisson orders, symplectic reflection
  algebras and representation theory}, J. Reine Angew. Math. \textbf{559}
  (2003), 193--216.

\bibitem[CM07]{conley2007}
Charles~H. Conley and Christiane Martin, \emph{Annihilators of tensor density
  modules}, J. Algebra \textbf{312} (2007), no.~1, 495--526.

\bibitem[Dix77]{dixmier1977ideaux}
Jacques Dixmier, \emph{Id\'eaux primitifs dans les alg\`ebres enveloppantes},
  J. Algebra \textbf{48} (1977), no.~1, 96--112.

\bibitem[Dix96]{dixmier1996enveloping}
\bysame, \emph{Enveloping algebras}, Graduate Studies in Mathematics, vol.~11,
  American Mathematical Society, Providence, RI, 1996, Revised reprint of the
  1977 translation.

\bibitem[Far00]{Farkas2000Poissonmodules}
Daniel~R. Farkas, \emph{Modules for {P}oisson algebras}, Comm. Algebra
  \textbf{28} (2000), no.~7, 3293--3306.

\bibitem[GW89]{goodearl1989introduction}
Ken~R. Goodearl and Robert~B. Warfield, Jr., \emph{An introduction to
  noncommutative {N}oetherian rings}, London Mathematical Society Student
  Texts, vol.~16, Cambridge University Press, Cambridge, 1989.

\bibitem[IS20]{iyudu2020enveloping}
Natalia~K. Iyudu and Susan~J. Sierra, \emph{Enveloping algebras with just
  infinite {G}elfand-{K}irillov dimension}, Ark. Mat. \textbf{58} (2020),
  no.~2, 285--306.

\bibitem[Jos74]{joseph1974ProofOT}
Anthony Joseph, \emph{Proof of the {G}elfand-{K}irillov conjecture for solvable
  {L}ie algebras}, Proc. Amer. Math. Soc. \textbf{45} (1974), 1--10.

\bibitem[Kal09]{Kaledin09Poisson}
Dmitry~B. Kaledin, \emph{Normalization of a {P}oisson algebra is {P}oisson},
  Tr. Mat. Inst. Steklova \textbf{264} (2009), 77--80.

\bibitem[Kir04]{Kirillov2004LecturesOT}
Alexandre~A. Kirillov, \emph{Lectures on the orbit method}, Graduate Studies in
  Mathematics, vol.~64, American Mathematical Society, Providence, RI, 2004.

\bibitem[KL00]{krause2000growth}
G\"unter~R. Krause and Thomas~H. Lenagan, \emph{Growth of algebras and
  {G}elfand-{K}irillov dimension}, revised ed., Graduate Studies in
  Mathematics, vol.~22, American Mathematical Society, Providence, RI, 2000.

\bibitem[Mat91]{mathieu1991bicon}
Olivier Mathieu, \emph{Bicontinuity of the {D}ixmier map}, J. Amer. Math. Soc.
  \textbf{4} (1991), no.~4, 837--863.

\bibitem[Moe80]{Moeglin1980ideaux}
Colette Moeglin, \emph{Id\'eaux bilat\`eres des alg\`ebres enveloppantes},
  Bull. Soc. Math. France \textbf{108} (1980), no.~2, 143--186.

\bibitem[MR01]{mcconnell2001noncommutative}
John~C. McConnell and J.~Chris Robson, \emph{Noncommutative {N}oetherian
  rings}, revised ed., Graduate Studies in Mathematics, vol.~30, American
  Mathematical Society, Providence, RI, 2001.

\bibitem[Oh99]{Oh99Poisson}
Sei-Qwon Oh, \emph{Poisson enveloping algebras}, Comm. Algebra \textbf{27}
  (1999), no.~5, 2181--2186.

\bibitem[Pha25]{pham2025orbit}
Tuan~A. Pham, \emph{The orbit method for the {V}irasoro algebra}, 2025,
  \href{https://arxiv.org/abs/2504.14670}{arXiv:2504.14670}.

\bibitem[Pha26]{phamthesis}
\bysame, \emph{Two-sided ideals of the enveloping algebras of the {V}irasoro
  and related {L}ie algebras}, Ph.D. thesis, University of Edinburgh, 2026,
  Available at \url{https://doi.org/10.7488/era/7765}.

\bibitem[PS23]{petukhov2022poisson}
Alexey~V. Petukhov and Susan~J. Sierra, \emph{The {P}oisson spectrum of the
  symmetric algebra of the {V}irasoro algebra}, Compos. Math. \textbf{159}
  (2023), no.~5, 933--984.

\bibitem[PT25]{PhamTimmins2025}
Tuan~A. Pham and James Timmins, \emph{The kernel and image of orbit
  homomorphisms for the {W}itt algebra}, 2025,
  \href{https://arxiv.org/abs/2510.00756}{arXiv:2510.00756}.

\bibitem[Rob72]{Robson1972}
J.~Chris Robson, \emph{Idealizers and hereditary {N}oetherian prime rings}, J.
  Algebra \textbf{22} (1972), 45--81.

\bibitem[RSS15]{Rogalski2015Sklyanin}
Daniel Rogalski, Susan~J. Sierra, and J.~Toby Stafford, \emph{Classifying
  orders in the {S}klyanin algebra}, Algebra Number Theory \textbf{9} (2015),
  no.~9, 2055--2119.

\bibitem[SW14]{sierra2014universal}
Susan~J. Sierra and Chelsea Walton, \emph{The universal enveloping algebra of
  the {W}itt algebra is not noetherian}, Adv. Math. \textbf{262} (2014),
  239--260.

\bibitem[SW16]{sierra2016maps}
\bysame, \emph{Maps from the enveloping algebra of the positive {W}itt algebra
  to regular algebras}, Pacific J. Math. \textbf{284} (2016), no.~2, 475--509.

\bibitem[Ush98]{Ushirobira1998}
Rosane Ushirobira, \emph{On the orbit method for the {L}ie algebra of vector
  fields on a curve}, J. Algebra \textbf{203} (1998), no.~2, 596--620.

\end{thebibliography}
\end{document}